\documentclass[11pt,a4paper]{scrartcl}
\pdfoutput=1
\usepackage[T1]{fontenc}
\usepackage[utf8x]{inputenc}
\usepackage[english]{babel}
\usepackage{lmodern}
\usepackage{amsfonts}
\usepackage{amsmath}
\usepackage{amsthm}
\usepackage{enumerate}
\usepackage{graphicx}
\usepackage{booktabs}

\newtheorem{theorem}{Theorem}[section]
\newtheorem{lemma}[theorem]{Lemma}
\newtheorem{remark}[theorem]{Remark}
\newtheorem{assumption}[theorem]{Assumption}

\newcommand{\abssec}[1]{\noindent\normalsize {\bf #1\quad }\ignorespaces}
\renewenvironment{abstract}{\abssec{Abstract}}{\par\vspace{.1in}}
\newenvironment{keywords}
   {\begin{trivlist}\item[]{\bfseries\sffamily Keywords:}\ }
   {\end{trivlist}}
\newenvironment{AMSMOS}{\abssec{AMS subject classification}}{\par\vspace{.1in}}

\numberwithin{equation}{section}

\newcommand{\dst}{\displaystyle}
\newcommand{\lnorm}[2]{\|#1\|_{L^2(#2)}}
\newcommand{\ds}{\,ds}
\newcommand{\into}[1]{\int\limits_{\Omega} #1 \,dx}
\newcommand{\intg}[1]{\int\limits_{\Gamma} #1 \,ds}
\newcommand{\scalar}[3]{(#1\,,\,#2)_{L^2(#3)}}

\renewcommand{\O}{\Omega}
\newcommand{\G}{\Gamma}

\newcommand{\yb}{\bar y}
\newcommand{\ub}{\bar u}
\newcommand{\pb}{\bar p}
\newcommand{\yhb}{\bar y_h}
\newcommand{\uhb}{\bar u_h}
\newcommand{\phb}{\bar p_h}

\newcommand{\WbO}[1]{W^{#1,2}_{\vec \beta}(\Omega)}
\newcommand{\mc}[3]{\multicolumn{#1}{#2}{#3}}

\title{Superconvergence for Neumann boundary control problems governed by semilinear elliptic equations}
\author{
J. Pfefferer\thanks{Universit\"at der Bundeswehr M\"unchen, Werner-Heisenberg-Weg 39, 85579 Neubiberg, Germany, \mbox{Johannes.Pfefferer@unibw.de}}
\and
K. Krumbiegel\thanks{Weierstrass Institute for Applied Mathematics and Stochastics, Nonlinear Optimization and Inverse Problems, Mohrenstrasse 39, D-10117 Berlin, \mbox{krumbiegel@wias-berlin.de}}}
\date{}

\begin{document}
\maketitle

\begin{abstract}
This paper is concerned with the discretization error analysis of semilinear Neumann boundary control problems in polygonal domains with pointwise inequality constraints on the control. The approximations of the control are piecewise constant functions. The state and adjoint state are discretized by piecewise linear finite elements. In a postprocessing step approximations of locally optimal controls of the continuous optimal control problem are constructed by the projection of the respective discrete adjoint state. Although the quality of the approximations is in general affected by corner singularities a convergence order of $h^2|\ln h|^{3/2}$ is proven for domains with interior angles smaller than $2\pi/3$ using quasi-uniform meshes. For larger interior angles mesh grading techniques are used to get the same order of convergence.
\end{abstract}

\begin{AMSMOS}
  65N30; 49K20, 49M25, 65N15, 65N50
\end{AMSMOS}

\section{Introduction}
In this paper we study discretization error estimates for the following Neumann boundary control problem governed by a semilinear elliptic partial differential equation:
\begin{gather}\label{costf}
\min F(y,u):=\displaystyle\frac{1}{2}\|y-y_d\|^2_{L^2(\O)}+\dst\frac{\nu}{2}\|u\|^2_{L^2(\G)}\\\label{stateeq}
\begin{aligned}
-\Delta y+d(x,y)&=0\quad\text{in }\Omega\\
\partial_n y&=u\quad\text{on }\Gamma
\end{aligned}\\\label{cconstr}
u\in U_{ad}:=\{u\in L^2(\Gamma):u_a\leq u\leq u_b \text{ a.e. on }\Gamma\}.
\end{gather}
In all what follows we denote the optimal control problem \eqref{costf}-\eqref{cconstr} by (P).
The precise conditions on all other given quantities in (P) are collected in the assumptions (A\ref{A2})-(A\ref{A6}) below.

We will discuss the full discretization of the optimal control problem combined with a postprocessing step, i.e., the state and the adjoint state are discretized by linear finite elements and the control by piecewise constant functions. Afterwards, approximations of locally optimal controls of the continuous optimal control problem are constructed which possess superconvergence properties.
This concept was established by Meyer and R\"osch in \cite{MeyerRoesch:2005} for linear-quadratic optimal control problems with distributed controls and a convergence order of $2$ in the $L^2(\O)$-norm was proven in convex domains. Using mesh grading techniques Apel, R\"osch and Winkler could prove in \cite{ApelRoeschWinkler:2007} the same convergence order for non convex polygonal domains. In a recent contribution by Mateos and R\"osch \cite{MateosRoesch:2008} this approach was extended to linear quadratic Neumann boundary control problems and a convergence rate of $\min(2,2-1/p)$ in the $L^2(\Gamma)$-norm was proven in convex domains with some $p$ satisfying $2<p<2\omega/(2\omega-\pi)$, where $\omega$ denotes the largest interior angle of the polygonal domain. 
Furthermore, for non convex domains a convergence rate of $1/2+\pi/\omega$ was proven.
This means that the convergence rate is lower than $3/2$ in the non convex case and decreases if the largest inner angle of the domain increases. Appropriately graded meshes in the neighborhood of the reentrant corners were used by Apel, Pfefferer and R\"osch in \cite{ApelPfeffererRoesch:2010} in order to prove an error bound of $ch^{3/2}$. This convergence rate was improved to $ch^2|\ln h|^{3/2}$ in a very recent contribution of Apel, Pfefferer and R\"osch \cite{ApelPfeffererRoesch:2012} using a new finite element error estimate on the boundary. Note, that only domains with interior angles larger than $2\pi/3$ need meshes which are appropriately graded to get this result. In the present work we combine the results derived in \cite{ApelPfeffererRoesch:2012} with techniques used in \cite{AradaCasasTroeltzsch:2002,CasasMateosTroeltzsch:2005} to prove optimal error estimates as discussed in~\cite{ApelPfeffererRoesch:2012} for the linear case.

Before we summarize the structure of the paper, let us give an overview on relevant literature concerning discretization of optimal control problems: we mention the contributions by Falk \cite{Falk:1973}, Geveci \cite{Geveci:1979}, Malanowski \cite{Malanowski:1982}, Arada, Casas and Tr\"oltzsch \cite{AradaCasasTroeltzsch:2002} and Casas, Mateos and Tr\"oltzsch \cite{CasasMateosTroeltzsch:2005} regarding the approximation by piecewise constant functions. For the usage of piecewise linear controls we refer to Casas and Tr\"oltzsch \cite{CasasTroeltzsch:2003}, Meyer and R\"osch \cite{MeyerRoesch:2005-2}, R\"osch \cite{Roesch:2006}, Casas and Mateos \cite{CasasMateos:2008} and the references therein. Convergence results and error estimates for elliptic optimal control problems governed by semilinear equations are especially derived in Arada, Casas and Tr\"oltzsch \cite{AradaCasasTroeltzsch:2002}, Casas, Mateos and Tr\"oltzsch \cite{CasasMateosTroeltzsch:2005} and Casas and Mateos \cite{CasasMateos:2008}.
For the variational discretization concept we refer to Hinze~\cite{Hinze:2005} in case of distributed control problems and to Casas and Mateos~\cite{CasasMateos:2008}, Mateos and R\"osch~\cite{MateosRoesch:2008}, Hinze and Matthes~\cite{HinzeMatthes:2008} and Apel, Pfefferer and R\"osch \cite{ApelPfeffererRoesch:2012} in case of Neumann boundary control problems. Using this concept in the context of linear elliptic Neumann boundary control problems one can achieve a discretization error bound of $ch^2|\ln h|^{3/2}$ on quasi-uniform meshes if the largest interior angle is smaller the $2\pi/3$. For larger interior angles one has to use appropriately graded meshes to deduce this result, cf. \cite{ApelPfeffererRoesch:2012}. In case of semilinear elliptic Neumann boundary control problems a convergence order of about $3/2$ is proven in~\cite{CasasMateos:2008} for this concept. But one can use the finite element error estimates on the boundary of~\cite{ApelPfeffererRoesch:2012} to derive the improved 
discretization error estimates as in the linear elliptic case.

The paper is organized as follows:
In Section 2 we introduce suitable weighted Sobolev space prescribing the regularity of solutions of elliptic boundary value problems. Moreover, we present first order necessary and second order sufficient optimality conditions for a local optimal solution of problem (P). Section 3 concerns the discretization of problem (P) and the establishment of a known uniform convergence results for solutions of the fully discretized problem to solutions of the continuous one. In Section 4 we elaborate several auxiliary results that are necessary in order to prove the superconvergence properties of the fully discrete counterpart of (P) in Section 5. Numerical experiments in the last section illustrate the proven results of the paper.

In the sequel $c$ denotes a generic constant which is always independent of the mesh parameter $h$.
\section{Optimality conditions and regularity results for problem (P)}
Throughout this paper let $\O$ be a bounded, two dimensional polygonal domain with Lipschitz boundary $\G$ and $m$ corner points $x^{(j)}$, $j=1,\dots,m$, counting counter-clockwise. In particular, $\G_j$ denotes the part of the boundary which connects the corners $x^{(j)}$ and $x^{(j+1)}$ except that $x^{(1)}$ is the intersection of $\bar\G_m$ and $\bar\G_1$. The angle between $\G_{j-1}$ and $\G_j$ is denoted by $\omega_j$ with the obvious modification for $\omega_1$.
Next, let us state some basic assumptions on the data of problem (P), which we require in the sequel.
\begin{assumption}\label{Psetting}
\begin{enumerate}[({A}1)]
\item \label{A2} The function $y_d\in C^{0,\sigma}(\bar\O)$ is given for some $\sigma>0$.
\item The regularization parameter $\nu>0$ and the bounds $u_a\le u_b$ are fixed real numbers.
\item \label{A4}The function $d=d(x,y): \Omega\times\mathbb{R}$ is measurable with respect
to $x\in \Omega$ for all fixed $y\in \mathbb{R}$, and twice continuously differentiable with
respect to $y$, for almost all $x \in\Omega$. Moreover,  we require $d(\cdot,0)\in L^2(\O),\,\frac{\partial d}{\partial y}(\cdot,0)\in C^{0,\sigma}(\bar \Omega) \text{with some }\sigma>0,\,\frac{\partial^2 d}{\partial y^2}(\cdot,0)\in L^\infty(\O)$ and 
\[
  \frac{\partial d}{\partial y}(x,y)\ge0\quad\text{for a.a. }x\in\O\text{ and }y\in\mathbb{R}.
\]
The derivatives of $d$ w.r.t. $y$ up to order two are uniformly Lipschitz on bounded sets, i.e. for all $M>0$
there exists $L_{d,M}>0$ such that $d$ satisfies
\begin{equation*}
\left\|\frac{\partial^2 d}{\partial
y^2}(\cdot,y_1)-\frac{\partial^2 d}{\partial
y^2}(\cdot,y_2)\right\|_{L^\infty(\Omega)}\le L_{d,M}|y_1-y_2|
\end{equation*}
for all $y_i \in \mathbb{R}$ with $|y_i|\le M$, $i=1,2$.
\item \label{A6} There is a subset $E_\O\subset\O$ of positive measure and a constant $c_\Omega>0$ such that $\frac{\partial d}{\partial y}(x,y)\geq c_\Omega$ in $E_\O\times \mathbb{R}$.
\end{enumerate}
\end{assumption}
To shorten the notation we will abbreviate $\frac{\partial d}{\partial y}$ and $\frac{\partial^2 d}{\partial y^2}$ by $d_y$ and $d_{yy}$, respectively.
Now let us begin with the study of the state equation. In general the regularity of the solution of a semilinear elliptic boundary value problem in polygonal domains is limited due to the appearance of corner singularities. If one uses classical Sobolev-Slobodetskij spaces $W^{s,p}(\O)$ to describe the regularity then this effect is reflected by the dependence of the parameters $s$ and $p$ on the size of the interior angles of the domain, compare e.g. \cite{Dauge:1988} and \cite{Grisvard:1985}. Instead, we will use weighted Sobolev spaces which incorporate better the singular behavior caused by the corners. The following exposition follows~\cite{ApelPfeffererRoesch:2012}. Let $\O_{R_j}$ and $\O_{R_j/2}$ be circular sectors which have the opening angle $\omega_j$ and the radii $R_j$ and $R_j/2$, respectively. These sectors are centered at the corners $x^{(j)}$ of the domain. The radii can be chosen arbitrarily with the only restriction that the circular sectors $\O_{R_j}$ do not overlap. The sides of the 
circular 
sectors $\O_{R_j}$ which coincide 
locally with the boundary $\G$ are denoted by $\G_j^+$ ($\varphi_j=\omega_j$) and $\G_j^-$ ($\varphi_j=0$) where $r_
j$ and $\varphi_j$ are the polar coordinates located at the corner point $x^{(j)}$. Furthermore, we define $\G_j^\pm=\G_j^+\cup\G_j^-$ and we set
\[
  \O^0=\O\backslash\bigcup_{j=1}^m\O_{R_j/2}\quad\text{and}\quad \G^0=\G\cap\bar\O^0.
\]
Next, we introduce the weighted Sobolev spaces. Let $\vec{\beta}=(\beta_1,\dots,\beta_m)^T$ be a real-valued vector. We define for $k\in\mathbb{N}_0$ and $p\in[1,\infty]$ the weighted Sobolev spaces $W^{k,p}_{\vec{\beta}}(\O)$ as the set of all functions on $\O$ with finite norm
\[
  \|v\|_{W^{k,p}_{\vec{\beta}}(\O)}=\|v\|_{W^{k,p}(\O^0)}+\sum\limits_{j=1}^m\|v\|_{W^{k,p}_{\beta_j}(\O_{R_j})},
\]
where the Sobolev spaces $W^{k,p}(\O)\ (=H^k(\O)\text{ for }p=2)$ are defined as usual. By means of standard multi-index notation the weighted parts are defined by
\begin{align*}
  \|v\|_{W^{k,p}_{\beta_j}(\O_{R_j})}&=\left(\sum\limits_{|\alpha|\le k}\|r_j^{\beta_j}D^\alpha v\|_{L^p(\O_{R_j})}^p\right)^{1/p}\quad \text{for } 1\leq p< \infty,\\
  \|v\|_{W^{k,\infty}_{\beta_j}(\O_{R_j})}&=\sum\limits_{|\alpha|\le k}\|r_j^{\beta_j}D^\alpha v\|_{L^\infty(\O_{R_j})}.
\end{align*}
For $k\ge1$ the corresponding trace spaces are denoted by $W^{k-1/p,p}_{\vec{\beta}}(\G)$ and the norm is given by
\[
  \|v\|_{W^{k-1/p,p}_{\vec{\beta}}(\G)}=\inf\{\|u\|_{W^{k,p}_{\vec{\beta}}(\O)}:\,u\in W^{k,p}_{\vec{\beta}}(\O),\,u|_{\G\setminus \mathcal{C}}=v \},
\]
where we denote with $\mathcal{C}$ the set of all corner points. Furthermore, we define the space $W^{k,p}_{\vec{\beta}}(\G)$ for $k\in \mathbb{N}_0$ and $p\in [1,\infty]$ by the norm
\[
  \|v\|_{W^{k,p}_{\vec{\beta}}(\G)}=\|v\|_{W^{k,p}(\G^0)}+\sum\limits_{j=1}^m\|v\|_{W^{k,p}_{\beta_j}(\G_j^\pm)}
\]
with
\begin{align*}
  \|v\|_{W^{k,p}_{\beta_j}(\Gamma_j^\pm)}&=\left(\sum_{|\alpha|\leq k}\left(\|r_j^{\beta_j}\partial_t^\alpha v\|_{L^p(\Gamma_j^+)}^p+\|r_j^{\beta_j}\partial_t^\alpha v\|_{L^p(\Gamma_j^-)}^p\right)\right)^{1/p} \quad \text{if } 1\leq p <\infty,\\
  \|v\|_{W^{k,\infty}_{\beta_j}(\Gamma_j^\pm)}&=\sum_{|\alpha|\leq k}\left(\|r_j^{\beta_j}\partial_t^\alpha v\|_{L^\infty(\Gamma_j^+)}+\|r_j^{\beta_j}\partial_t^\alpha v\|_{L^\infty(\Gamma_j^-)}\right).
\end{align*}
Note, that $\partial_t v$ denotes the tangential derivative of $v$. The semi-norms
\[|\cdot|_{W^{k,p}_{\vec\beta}(\Omega)} \text{ and } |\cdot|_{W^{k,p}_{\vec\beta}(\Gamma)}\]
are analogously defined to the classical Sobolev semi-norms.

As usual, we denote with $C^k(\bar \O)$ the set of all functions on $\O$ with bounded and uniformly continuous derivatives up to order $k$. The H\"older space $C^{k,\sigma}(\bar \O)$ additionally possesses bounded derivatives of order $k$ which are H\"older continuous with H\"older exponent $\sigma\in(0,1]$.

We proceed with studying regularity results concerning linear and semilinear elliptic partial differential equations in classical and weighted Sobolev spaces.
\begin{lemma}\label{pro:reglinearPDE}
Let $E_\Omega$ be a subset of $\Omega$ with $|E_\Omega|>0$ and let $m,M$ be a constants greater than zero. Moreover, let $\alpha$ be a function in $L^\infty(\O)$ with $\alpha(x)\geq 0$ for a.a. $x \in \Omega$, $\alpha(x)\geq m$ for a.a. $x\in E_\Omega$ and $\|\alpha\|_{L^\infty(\Omega)}\leq M$. Then the problem 
\begin{equation}\label{linPDE}
	\begin{aligned}
	-\Delta \phi+\alpha\phi&=f\quad\text{in }\Omega\\
	\partial_n \phi&=g\quad\text{on }\Gamma
	\end{aligned}
\end{equation}
admits a unique solution $\phi$ in
\begin{enumerate}[(i)]
  \item \label{linReg1} $H^{3/2}(\O)$ for $f\in L^2(\Omega)$ and $g\in L^2(\Gamma)$.
  \item \label{linReg2} $W^{2,2}_{\vec{\beta}}(\O)$ for $f\in W^{0,2}_{\vec{\beta}}(\O)$ and $g\in W^{1/2,2}_{\vec{\beta}}(\G)$ where $\beta_j$ has to satisfy
	\begin{equation}\label{condbeta}
		1>\beta_j>\max(0,1-\lambda_j) \quad \text{or} \quad \beta_j=0 \text{ and } 1-\lambda_j<0
	\end{equation}
	with $\lambda_j=\pi/\omega_j$ for $j=1,\dots,m$. Furthermore, the a priori estimate
	\[
	  \|\phi\|_{W^{2,2}_{\vec\beta}(\O)}\leq c \left(\|f\|_{W^{0,2}_{\vec\beta}(\O)}+\|g\|_{W^{1/2,2}_{\vec\beta}(\G)}\right)
	\]
	is valid with a constant $c$ which may depend on $m$ and $M$ but is independent of $\alpha$.
\end{enumerate}
\end{lemma}
\begin{proof}
(i) The existence of a unique solution $\phi\in H^1(\Omega)$ of~\eqref{linPDE} can be deduced from the Lax-Milgram Theorem. The first assertion is then a consequence of~\cite{JerisonKenig:1981} or~\cite[Lemma 2.2]{CasasMateosTroeltzsch:2005} since $f-\alpha\phi$ belongs to $L^2(\Omega)$.

\noindent (ii) The functional
\[
  F(v)=\into{fv} + \intg{gv}
\]
is continuous on $H^1(\Omega)$ for arbitrary $f\in W^{0,2}_{\vec{\beta}}(\O)$ and $g\in W^{1/2,2}_{\vec{\beta}}(\G)$ with $\beta_j$ satisfying~\eqref{condbeta} for $j=1,\ldots,m$, cf. \cite[Lemma 6.3.1]{MazyaRossmann:2010}. Thus, the existence of a unique solution $\phi\in H^1(\Omega)$ of~\eqref{linPDE} is again given by the Lax-Milgram Theorem. Furthermore, there holds, independent of $\alpha$,
\begin{align*}
  \|f+(1-\alpha)\phi\|_{W^{0,2}_{\vec\beta}(\Omega)}&\leq \|f\|_{W^{0,2}_{\vec{\beta}}(\O)}+c\|\phi\|_{L^2(\Omega)}\leq\|f\|_{W^{0,2}_{\vec{\beta}}(\O)}+c\|\phi\|_{H^1(\Omega)}\\
  &\leq c\left(\|f\|_{W^{0,2}_{\vec{\beta}}(\O)}+\|g\|_{W^{1/2,2}_{\vec{\beta}}(\G)}\right).
\end{align*}
Combining this with~\cite[Lemma 2.4]{ApelPfeffererRoesch:2012} yields the second assertion.
\end{proof}

\begin{lemma}\label{genPDEsol}
Let the Assumptions (A\ref{A4})-(A\ref{A6}) be satisfied. Then the problem
\begin{equation*}
\begin{aligned}
-\Delta \phi+d(x,\phi)&=f\quad\text{in }\Omega\\
\partial_n \phi&=g\quad\text{on }\Gamma
\end{aligned}
\end{equation*}
has a unique solution $\phi$ which belongs to
\begin{enumerate}[(i)]
  \item \label{Reg1} $H^{3/2}(\O)$ for $f\in L^2(\Omega)$ and $g\in L^2(\Gamma)$.
  \item \label{Reg2} $W^{2,2}_{\vec{\beta}}(\O)$ for $f\in W^{0,2}_{\vec{\beta}}(\O)$ and $g\in W^{1/2,2}_{\vec{\beta}}(\G)$ with $\beta_j$ satisfying~\eqref{condbeta}.
\end{enumerate}
\end{lemma}
\begin{proof}
(i) Due to the Assumptions (A\ref{A4})-(A\ref{A6}), it is classical to show the existence of a unique solution $\phi\in H^1(\O)\cap C^0(\bar\O)$ for right hand sides $f\in L^2(\O)$ and $g\in L^2(\G)$, see e.g. \cite{Casas:1993}. 
We deduce by Assumption (A\ref{A4}) that $f-d(\cdot,\phi)+\phi$ belongs to $L^2(\Omega)$. The assertion follows now from Lemma~\ref{pro:reglinearPDE}~(\ref{linReg1}) with $\alpha\equiv 1$.

\noindent(ii) First, we observe that there exist $r,t>1$ such that
\[
	W^{0,2}_{\vec\beta}(\Omega)\hookrightarrow L^r(\Omega)\quad \text{and}\quad W^{1/2,2}_{\vec\beta}(\G)\hookrightarrow L^t(\G)
\]
provided that $\beta_j$ satisfies~\eqref{condbeta} for $j=1,\ldots,m$, cf.~\cite{Rossmann:1988}. Thus the existence of a unique solution $\phi\in H^1(\O)\cap C^0(\bar\O)$ is again given by~\cite{Casas:1993}. Since $H^{1}(\O)\hookrightarrow L^2(\O)\hookrightarrow W^{0,2}_{\vec{\beta}}(\O)$ for $\beta_j$ satisfying~\eqref{condbeta} we obtain $f-d(\cdot,\phi)+\phi\in W^{0,2}_{\vec{\beta}}(\O)$ from~(A\ref{A4}). According to Lemma~\ref{pro:reglinearPDE}~(\ref{linReg2}) we can conclude the stated regularity if we set $\alpha\equiv 1$.
\end{proof}
Based on the last lemma we can introduce the control-to-state operator
\begin{equation}\label{G}
	G: L^2(\G)\rightarrow H^{3/2}(\Omega),\, G(u)=y,
\end{equation}
that assigns to every control $u$ the unique solution $y$ of the state equation \eqref{stateeq}. By this we can reformulate problem (P) and we obtain its reduced formulation
\[
  \min_{u\in U_{ad}} J(u):=F(G(u),u)=\displaystyle\frac{1}{2}\|G(u)-y_d\|^2_{L^2(\O)}+\dst\frac{\nu}{2}\|u\|^2_{L^2(\G)}.
\]
To indicate the dependence of the state $y$ on the control $u$ we will also write $y(u)$ for $G(u)$ in the sequel. Note that an extension of the control-to-state operator to the previously defined weighted Sobolev spaces is not necessary for the formulation of classical optimality conditions for problem (P). But first, let us discuss the differentiability properties of the control-to-state mapping.
\begin{theorem}\label{diffG}
Let the Assumptions (A\ref{A4})-(A\ref{A6}) be satisfied. Then the mapping $G:L^2(\G)\rightarrow H^{3/2}(\Omega)$, defined by~\eqref{G} is of class
$C^2$. Moreover, for all $u,\,v\in L^2(\G)$, $y_v=G'(u)v$ is defined as the
solution of
\begin{equation*}
\begin{aligned}
-\Delta y_v+d_y(x,y)y_v&=0\quad \mbox{in }\O\\
\partial_{n}y_v&=v\quad\mbox{on }\G
\end{aligned}
\end{equation*}
Furthermore, for every $v_1,v_2\in L^2(\G)$, $y_{v_1,v_2}=G''(u)[v_1,v_2]$ is
the solution of
\begin{equation*}
\begin{aligned}
-\Delta y_{v_1,v_2}+d_y(x,y)y_{v_1,v_2}&=-d_{yy}(x,y)y_{v_1}y_{v_2}&&\quad
\mbox{in }\O\\
\partial_{n}y_{v_1,v_2}&=0&&\quad\mbox{on }\G,
\end{aligned}
\end{equation*}
where $y_{v_i}=G'(u)v_i,\,i=1,2$.
\end{theorem}
The proof of this theorem is based on the implicit function theorem. It can be found in \cite[Theorem 3.1]{CasasTroeltzsch:2012}. We also refer to \cite{CasasMateos:2002}, \cite{CasasMateosTroeltzsch:2005} and \cite{Troeltzsch:2010}. The next theorem is devoted to the first order optimality conditions and regularity results for locally optimal solutions of problem (P).
\begin{theorem}\label{firstorder}
Let Assumption \ref{Psetting} be fulfilled. Then problem (P) admits at least one solution in $U_{ad}$. For every (local) solution $\ub\in U_{ad}$ of problem (P) there exists a unique optimal state $\yb\in H^{3/2}(\O)$ and optimal adjoint state $\pb\in H^{3/2}(\O)$ such that
\begin{gather}\label{optstate}
\begin{aligned}
-\Delta \yb+d(x,\yb)&=0&&\quad\text{in }\Omega\\
\partial_n \yb&=\ub&&\quad\text{on }\Gamma
\end{aligned}\\\label{optadjoint}	
\begin{aligned}
-\Delta\pb +d_y(x,\yb)\pb&=\yb-y_d&&\quad \mbox{in }\O\\
\partial_{n}\pb&=0&&\quad\mbox{on }\G
\end{aligned}\\\label{varineq}
J'(\bar u)(u-\bar u)=\scalar{\pb+\nu\ub}{u-\ub}{\G}\geq 0\quad\forall u\in U_{ad}.
\end{gather}
Moreover, let $\beta_j$ satisfy~\eqref{condbeta} for $j=1,\ldots,m$. Then there holds
for $\epsilon<\min(1,\min_j(1-\beta_j))$ that $\yb,\,\pb\in \WbO{2}\cap C^{0,\epsilon}(\bar\Omega)$, $\pb|_\G\in H^1(\G)\cap C^{0,\epsilon}(\G)$ and $\ub\in H^1(\G)\cap C^{0,\epsilon}(\G)$.
\end{theorem}
\begin{proof}
Due to the structure the reduced cost functional of problem \eqref{costf}-\eqref{cconstr} is of class $C^2$ from $L^2(\G)$ to $\mathbb{R}$, cf. \cite[Theorem 3.2 and Remark 3.3]{CasasTroeltzsch:2012}. The convexity of the cost functional with respect to the control $u$ implies the existence of at least one solution of problem (P) in $U_{ad}$ under Assumption \ref{Psetting}, which can be shown by standard arguments, see e.g. \cite[Section 4.4.2]{Troeltzsch:2010}. The first order optimality conditions \eqref{optstate}-\eqref{varineq} are based on Theorem \ref{diffG} and can also be derived by standard arguments, see e.g. \cite[Section 4.6]{Troeltzsch:2010}.
It remains to prove the regularity assertion. By means of Lemma~\ref{genPDEsol}~(\ref{Reg1}) we get $\bar y\in H^{3/2}(\O)$ for every (local) solution $\bar u \in U_{ad}\subset L^2(\G)$. Since $\bar y$ belongs to $H^{3/2}(\O)\hookrightarrow L^\infty(\O)$ and $y_d$ to $C^{0,\sigma}(\bar \O)\hookrightarrow L^2(\O)$ we can conclude with Assumption~\ref{Psetting} and Lemma~\ref{pro:reglinearPDE}~(\ref{linReg1}) that the adjoint state $\pb$ is an element of $H^{3/2}(\O)$. Using Lemma~\ref{pro:reglinearPDE}~(\ref{linReg2}) we even obtain that $\pb$ belongs to $\WbO{2}$ if $\vec \beta$ satisfies~\eqref{condbeta}. According to \cite[Lemma 2.1]{ApelPfeffererRoesch:2012} the embedding
\begin{equation}\label{eq:W2q}
W^{2,2}_{\vec\beta}(\O)\hookrightarrow W^{2,q}(\O)
\end{equation}
is valid for $q<\min(2,\min_j(2/(\beta_j+1)))$. Since $\lambda_j=\pi/\omega_j>1/2$, we can conclude, that there is a $\beta_j$ such that $1/2>\beta_j>1-\lambda_j$, which allows the choice $q=4/3$ in~\eqref{eq:W2q}. Thus, Theorem 4.11 of \cite{Necas:2012} implies $\bar p\in H^1(\Gamma)$.
Moreover, it is well known that the variational inequality \eqref{varineq} is equivalent to the projection formula
\begin{equation}\label{projform}
  \ub=\Pi_{[u_a,u_b]}\left(-\frac{1}{\nu}\pb\right)\text{ for a.a. } x\in\G.
\end{equation}
with $\Pi_{[u_a,u_b]}f(x):=\max(u_a,\min(u_b,f(x)))$. Hence, the local optimal control $\ub$ belongs to $H^1(\G)$, cf. \cite[Theorem A.1]{Kinderlehrer:1980}. Furthermore, there are the embeddings $H^1(\Gamma)\hookrightarrow H^{1/2}(\G)\hookrightarrow W^{1/2,2}_{\vec \beta}(\G)$ for $\beta_j\geq 0$. Thus, we can conclude $y\in\WbO{2}$ for $\vec \beta$ satisfying~\eqref{condbeta} by means of Lemma \ref{genPDEsol}~(\ref{Reg2}). Finally, the embedding \eqref{eq:W2q} and the Sobolev inequality imply $\bar y, \bar p\in C^{0,\epsilon}(\bar \O)$ and $\bar u\in C^{0,\epsilon}(\Gamma)$ if ${\epsilon<\min(1,\min_j(1-\beta_j))}$.
\end{proof}
Actually, the proof of Theorem~\ref{firstorder} requires only $y_d\in L^2(\Omega)$. Due to the additional assumption $y_d\in C^{0,\sigma}(\bar\O)$ the regularity of the adjoint state $\pb$ can be increased. This fact is essential for improved finite element error estimates on the boundary and for the main result of this paper.
\begin{theorem}\label{regadjoint}
Let Assumption~\ref{Psetting} be satisfied. Furthermore, let $\beta_j$ and $\gamma_j$ satisfy the conditions
\begin{align}
	1/2>\beta_j>\max(0,3/4-\lambda_j/2)\quad &\text{or}\quad \beta_j=0\text{ and }3/4-\lambda_j/2< 0,\notag\\
  2>\gamma_j>\max(0,2-\lambda_j)\quad &\text{or}\quad \gamma_j=0\text{ and }2-\lambda_j< 0\label{eq:gammab}
\end{align}
with $\lambda_j=\pi/\omega_j$ for $j=1,\ldots,m$. Then the adjoint state $\pb$ satisfying the adjoint equation \eqref{optadjoint} belongs to $W^{2,\infty}_{\vec \gamma}(\O)$ and its restriction to the boundary $\pb_{|\Gamma}$ to $W^{2,2}_{2\vec\beta}(\Gamma)\hookrightarrow W^{1,\infty}_{\vec\beta}(\Gamma)$.
\end{theorem}
\begin{proof}
According to Theorem~\ref{firstorder} there is a $\epsilon>0$ such that $\bar y$ and $\bar p$ belong to $C^{0,\epsilon}(\bar \O)$ and hence $\bar y-y_d+(1-d_y(\cdot,\bar y))\bar p$ either having regard to Assumption~(A\ref{A4}). Therefore, Lemma 2.6 of \cite{ApelPfeffererRoesch:2012} implies $\pb\in W^{2\infty}_{\vec \gamma}(\O)$ and $\pb_{|\G}\in W^{2\infty}_{\vec \gamma}(\G)$ if $\vec \gamma$ satisfies~\eqref{eq:gammab}. The stated regularity on the boundary is then a consequence of the Sobolev inequality and embeddings in weighted Sobolev spaces, cf. \cite[Corollary 4.2]{ApelPfeffererRoesch:2012}.
\end{proof}

For the statement of second order sufficient optimality conditions we will count on so called strongly active sets. We start with the definition of the $\tau$-critical cone associated to a control $\bar u$:
\begin{equation}\label{criticalcone}
	C_\tau(\ub):=\{v\in L^2(\G):\,v\text{ satisfies }\eqref{strongactive}\},
\end{equation}
where
\begin{equation}\label{strongactive}
	v(x)\left\{
	\begin{aligned}
	\ge 0,&\quad\text{if }\ub(x)=u_a\\
	\le 0,&\quad\text{if }\ub(x)=u_b\\
	=0, &\quad\text{if }|\pb(x)+\nu\ub(x)|>\tau.
	\end{aligned}
	\right.
\end{equation}
Furthermore, straightforward computations using Theorem~\ref{diffG} yield the following well known formulation of the second derivative of the reduced cost functional $J(u)$:
\begin{equation*}
  J''(u)[v_1,v_2]=\into{y_{v_1}y_{v_2}-p(y(u))d_{yy}(x,y(u))y_{v_1}y_{v_2}}+\intg{\nu v_1v_2}
\end{equation*}
with $p(y(u))$ being the solution of~\eqref{optadjoint} with $\bar y$ replaced by $y(u)$. Now, we are in the position to formulate second order sufficient optimality conditions.
\begin{theorem}\label{suffcond}
Let Assumption~\ref{Psetting} be satisfied. Moreover, let $\ub\in U_{ad}$ be a control satisfying the first order optimality conditions given in Theorem \ref{firstorder}.   Further, it is assumed that there are two constants $\tau>0$ and $\delta>0$ such that
\begin{equation}\label{SSC}
  J''(\ub)[v,v]\ge\delta\|v\|_{L^2(\G)}^2 
\end{equation}
for all $v\in C_\tau(\ub)$. Then, there exist $\beta>0$ and $\varrho>0$ such that
\[
  J(u)\ge J(\ub)+\beta\|u-\ub\|_{L^2(\G)}^2
\]
is satisfied for every $u\in U_{ad}$ with $\lnorm{u-\ub}{\G}\le\varrho$.
\end{theorem}
\begin{proof}
For details regarding the proof of the theorem we refer to e.g. \cite[Corollary 3.6]{CasasTroeltzsch:2012}, see also \cite{BonnansZidani:1999}, \cite{CasasMateos:2002}, \cite[Chapter 4.10]{Troeltzsch:2010}, and the references therein. Note, that we do not have to deal with the two-norm discrepancy due to the special structure of the optimal control problem, cf. the general setting in \cite[Section 3]{CasasTroeltzsch:2012}.
\end{proof}

\section{Discretization and fully discrete approximation of (P)}\label{sec:discretization}
Here, we define a finite element based approximation of the optimal control problem (P). To this end, we introduce a family of graded triangulations $\mathcal T_h$ of $\O$ in the sense of Ciarlet \cite{Ciarlet:1991}, where $h$ denotes the global mesh parameter, which is assumed to be less than $1$. Note, that there is a segmentation $\mathcal{E}_h$ of the boundary $\Gamma$ induced by the triangulation $\mathcal{T}_h$. The vector $\vec\mu\in\mathbb{R}^m$ summarizes the grading parameters $\mu_j\in(0,1]$, $j=1,\ldots,m$, regarding the corner points $x^{(j)}$. The distances of the triangle $T\in\mathcal{T}_h$ and edge $E\in\mathcal{E}_h$ to the corner $x^{(j)}$ are defined by $r_{T,j}:=\inf_{x\in T}|x-x^{(j)}|$ and $r_{E,j}:=\inf_{x\in E}|x-x^{(j)}|$, respectively.
We assume that the mesh size $h_T$ of a triangle $T\in\mathcal{T}_h$ satisfies
\begin{equation}\label{mesh}
\begin{aligned}
c_1h^{1/\mu_j}&\le h_T\le c_2 h^{1/\mu_j}&&\quad\text{for }r_{T,j}=0,\\
c_1hr_{T,j}^{1-\mu_j}&\le h_T\le c_2hr_{T,j}^{1-\mu_j}&&\quad\text{for }0<r_{T,j}\le R_j,\\
c_1h&\le h_T\le c_2h&&\quad\text{for }r_{T,j}> R_j
\end{aligned}
\end{equation}
for $j=1,\ldots,m$ with the radii $R_j$ that has been introduced in the beginning of Section 2. As a consequence there holds for the mesh size $h_E$ of an element $E\in \mathcal{E}_h$ being an edge of the triangle $T\in \mathcal{T}_h$
\[
  h_E\sim h_T\quad \forall E\subset \bar T.
\]
Furthermore, we introduce for $j=1,\ldots,m$ the sub-triangulations $\mathcal{E}_{h,j}$ of $\mathcal{E}_h$ satisfying $\bigcup_{E\in \mathcal{E}_{h,j}}\bar E\subset \Gamma^\pm_{j}$ and $E\cap\Gamma^\pm_j\neq E$ for all $E\notin\mathcal{E}_{h,j}$. We define $\mathcal{E}_{h,0}=\mathcal{E}_h\backslash\bigcup_{j=1}^m\mathcal{E}_{h,j}$.
Associated with this triangulation we set
\begin{align*}
V_h&:=\left\{y_h\in C^0(\bar\O):\,y_h|_T\in\mathcal P_1(T)\,\,\forall T\in\mathcal T_h\right\}\\
U_h&:=\left\{u_h\in L^\infty(\Gamma):\,u_h|_E\in\mathcal P_0(E)\,\,\forall E\in\mathcal E_h\right\}\\
U_{ad,h}&:=U_h\cap U_{ad}, 
\end{align*}
where $\mathcal P_1(T)$ and $\mathcal{P}_0(E)$ denote the spaces of all polynomials of degree less than or equal $1$ on $T$ or $0$ on $E$, respectively. Next, we introduce the discrete counterpart to the control-to-state operator $G$ in \eqref{G}. For each $u\in L^2(\G)$, we denote by $y_h(u)=G_h(u)$ the unique element of $V_h$ that satisfies
\begin{equation}\label{vardisPDE}
a(y_h(u),v_h)+\into{d(x,y_h(u))v_h}=\intg{uv_h}\quad\quad\forall v_h\in V_h
\end{equation}
with the bilinear form
\[
  a:H^1(\Omega)\times H^1(\Omega)\rightarrow \mathbb{R},\,\, a(y,v)=\into{\nabla y\cdot\nabla v}.
\]
The existence and uniqueness of a solution of \eqref{vardisPDE} can be deduced in a standard way using the monotonicity of $d$. Then the fully discretized version (P$_h$) of the optimal control problem (P) reads as follows
\begin{gather*}
\min_{u_h\in U_{ad,h}}  J_h(u_h):=\displaystyle\frac{1}{2}\|G_h(u_h)-y_d\|^2_{L^2(\O)}+\dst\frac{\nu}{2}\|u_h\|^2_{L^2(\G)}.
\end{gather*}
Since the cost functional $J_h$ is continuous and the admissible set compact, the existence of at least one solution of problem (P$_h$) is given. The first order optimality conditions can be written by
\begin{theorem}
Let Assumption~\ref{Psetting} be satisfied. Furthermore, let $\uhb\in U_{ad,h}$ be a local optimal solution of (P$_h$). Then there exist a discrete optimal state $\yhb\in V_h$ and a discrete optimal adjoint state $\phb\in V_h$ such that
\begin{gather}
a(\yhb,v_h)+\into{d(x,\yhb)v_h}=\intg{\uhb v_h}\quad\forall v_h\in V_h,\\\label{disadjoint}
a(\phb,v_h)+\into{d_y(x,\yhb)\phb v_h}=\into{(\yhb-y_d) v_h}\quad\forall v_h\in V_h,\\\label{varineqdis}
J_h'(\bar u_h)(u_h-\bar u_h)=\scalar{\phb+\nu\uhb}{u_h-\uhb}{\G}\ge0\quad\quad\forall u_h\in U_{ad,h}.
\end{gather}
\end{theorem}
For the sake of completeness the second derivative of the cost functional of the fully discretized problem (P$_h$) can be formulated by:
\begin{equation*}
  J_h''(u_h)[v_1,v_2]=\into{y_h^{v_1}y_h^{v_2}-p_h(y_h(u_h))d_{yy}(x,y_h(u_h))y_h^{v_1}y_h^{v_2}}+\intg{\nu v_1v_2},
\end{equation*}
where $p_h(y_h(u_h))$ is the solution of the adjoint equation \eqref{disadjoint} w.r.t. $y_h(u_h)$ and $y_h^{v_i},\,i=1,2$ is the solution of the linearized discrete state equation with respect to $v_i\in L^2(\Gamma)$, i.e.,
\begin{equation}\label{Gh'}
  a(y_h^{v_i},v_h)+\into{d_y(x,y_h(u_h))y_h^{v_i} v_h}=\intg{v_i v_h}\quad\forall v_h\in V_h.
\end{equation}
For the purpose of a compact notation let us set $\vec \lambda=(\lambda_1,\ldots,\lambda_m)=(\pi/\omega_1,\ldots,\pi/\omega_m)$ and $\vec a=(a,\ldots,a)\in \mathbb{R}^m$ for any $a\in\mathbb{R}$, e.g. $\vec 1=(1,\ldots,1)\in\mathbb{R}^m$. Furthermore, all inequalities involving vectorial parameters must be understood component-by-component. The following lemma is related to finite element error estimates for linear elliptic PDEs on quasi-uniform and graded meshes that will be useful in the sequel.
\begin{lemma}\label{feerrorgenPDE}
Let $\phi$ be the solution of~\eqref{linPDE} and $\phi_h\in V_h$ be the solution of
\begin{gather*}
a(\phi_h,v_h)+\into{\alpha\phi_h v_h}=\into{fv_h}+\intg{gv_h}\quad\forall v_h\in V_h
\end{gather*}
with $\alpha$ being the function introduced in Lemma~\ref{pro:reglinearPDE}. Then the following assertions hold:
\begin{itemize}
\item[(i)] Let $\vec \mu <\vec\lambda$, $f\in W^{0,2}_{\vec 1-\vec\mu}(\Omega)$ and $g\in W^{1/2,2}_{\vec 1-\vec\mu}(\G)$. Then the error estimates
\begin{equation}\label{eq:feerrorgenPDE}
\lnorm{\phi-\phi_h}{\O}+h\|\phi-\phi_h\|_{H^1(\O)}\le ch^2(\|f\|_{W^{0,2}_{\vec 1-\vec\mu}(\Omega)}+\|g\|_{W^{1/2,2}_{\vec 1-\vec\mu}(\G)})
\end{equation}
hold independent of $\alpha$.
\item [(ii)] Let $\sigma\in(0,1]$ and $M\geq 0$ be given and let $\alpha$ additionally belong to $C^{0,\sigma}(\bar \Omega)$ with $\|\alpha\|_{C^{0,\sigma}(\bar \Omega)}\leq M$. Moreover, let $f\in C^{0,\sigma}(\bar\O)$ and $g\equiv 0$. Then the finite element error on the boundary admits independently of $\alpha$ the estimate
\begin{equation}\label{eq:feerrorgenPDEbound}
 \lnorm{\phi-\phi_h}{\G}\le ch^2|\ln h|^{3/2}\|f\|_{C^{0,\sigma}(\bar\O)} 
\end{equation}
provided that the mesh grading parameters satisfy $\vec1/4<\vec\mu<\vec1/4+\vec\lambda/2$.
\end{itemize}
\end{lemma}
\begin{proof}
For the proof of (i) and (ii) we refer to \cite[Lemma 4.1]{ApelPfeffererRoesch:2010} and \cite[Theorem 3.2]{ApelPfeffererRoesch:2012}, respectively. In both papers $\alpha\equiv 1$ is assumed, but it can be extended to the more general case in a natural way.
\end{proof}
Before we are in the position to deal with the superconvergence properties of the fully discrete optimal control problem (P$_h$), we have to ensure that every local minimum $\ub$ of (P) can be approximated by a local minimum of (P$_h$) provided that $\ub$ satisfies the second order sufficient optimality conditions. But first we need to determine the order of convergence of the solution of the discrete state equation \eqref{vardisPDE} to the solution of the continuous state equation \eqref{stateeq}. An analogous result is of course needed for the adjoint equation. Forthcoming, we will denote with $p(y)$ and $p_h(y)$ the solution of~\eqref{optadjoint} and~\eqref{disadjoint} with $\bar y$ and $\bar y_h$ replaced by $y\in L^\infty(\Omega)$, respectively. Note, that $y(u)=G(u)$ and $y_h(u)=G_h(u)$.
\begin{theorem}\label{errorestimates}
Let Assumption~\ref{Psetting} be satisfied. Then there holds:
\begin{itemize}
 \item[(i)]  For $\vec \mu<\vec\lambda$ and $u\in W^{1/2,2}_{\vec 1-\vec\mu}(\G)$ the discretization error estimates
\begin{align*}
\|y(u)-y_h(u)\|_{L^2(\O)}+h\|y(u)-y_h(u)\|_{H^1(\O)}\le ch^2\\
\|p(y(u))-p_h(y(u))\|_{L^2(\O)}+h\|p(y(u))-p_h(y(u))\|_{H^1(\O)}\le ch^2
\end{align*}
are valid.
\item[(ii)] For $u\in L^2(\G)$ there is a $\varepsilon>0$ arbitrarily small such that 
\begin{align}\label{fe-error-inf}
 \|y(u)-y_h(u)\|_{L^\infty(\O)}+\|p(y(u))-p_h(y(u))\|_{L^\infty(\O)}\le ch^{1/2-\varepsilon}.
\end{align}
\item[(iii)] For every $u_1,u_2\in L^2(\G)$ and $y_1,y_2\in L^\infty(\O)$ there holds
\begin{align*}
\|y(u_1)-y(u_2)\|_{H^1(\O)}+\|y_h(u_1)-y_h(u_2)\|_{H^1(\O)}&\le c\|u_1-u_2\|_{L^2(\G)},\\
\|p(y_1)-p(y_2)\|_{H^1(\O)}+\|p_h(y_1)-p_h(y_2)\|_{H^1(\O)}&\le c\|y_1-y_2\|_{L^2(\O)}.
\end{align*}
\item[(iv)]
Moreover, if $u_h\rightharpoonup u$ weakly in $L^2(\G)$, then $y_h(u_h)\rightarrow y(u)$ and $p_h(y_h(u_h))\rightarrow p(y(u))$ strongly in $C^0(\bar\O)$.
\end{itemize}
\end{theorem}
\begin{proof}We will prove the theorem for the state. The corresponding proof for the adjoint state can either be done analogously using the estimates for the states where required or is simply a consequence of Lemma~\ref{feerrorgenPDE}.

(i) Due to Assumption (A\ref{A4}) a generalization of Cea's Lemma to semilinear elliptic partial differential equations is available, cf. \cite[Lemma 2, Theorem 2]{CasasMateos:2002-2}. In particular we have
\begin{equation}\label{eq:cea}
  \|y(u)-y_h(u)\|_{H^1(\O)}\le c\inf_{v_h\in V_h}\|y-v_h\|_{H^1(\O)}.
\end{equation}
By means of Lemma~\ref{genPDEsol}~(\ref{Reg2}) we derive analogously to \cite[Lemma 4.1]{ApelPfeffererRoesch:2010}
\[
\|y(u)-y_h(u)\|_{H^1(\O)}\le ch\|y\|_{W^{2,2}_{\vec 1-\vec\mu}(\Omega)}\le ch
\]
for mesh grading parameters $\vec \mu<\vec\lambda$. Following the lines of \cite[Lemma 4]{CasasMateos:2002-2} and \cite[Lemma 4.1]{ApelPfeffererRoesch:2010} one can double the order of convergence in the $L^2(\O)$-norm.

(ii) For a control $u\in L^2(\G)$ we can only assure that the state belongs to the space $H^{3/2}(\O)$, cf. Lemma~\ref{genPDEsol}~(\ref{Reg1}). We proceed with
\[
  \|y(u)-y_h(u)\|_{L^\infty(\O)}\le \|y(u)-I_hy(u)\|_{L^\infty(\O)}+\|I_hy(u)-y_h(u)\|_{L^\infty(\O)},
\]
where $I_h$ denotes the classical nodal interpolation operator. Note that $I_hy(u)$ is well defined due to the embedding $H^{3/2}(\O)\hookrightarrow C^0(\bar\O)$. Using standard techniques of interpolation error estimates, the first term can be estimated by
\[
  \|y(u)-I_hy(u)\|_{L^\infty(\O)}\le ch^{1/2}|y(u)|_{H^{3/2}(\O)}.
\]
Next, we assume that $|I_hy(u)-y_h(u)|$ admits its maximum in an element $T^*$. By means of an inverse estimate, the embedding $H^1(\O)\hookrightarrow L^p(\O)$ $(p<\infty)$,~\eqref{eq:cea} and standard interpolation error estimates we derive
\begin{align*}
\|I_hy(u)-y_h(u)\|_{L^\infty(\O)}&= \|I_hy(u)-y_h(u)\|_{L^\infty(T^*)}\le ch_{T^*}^{-2/p} \|I_hy(u)-y_h(u)\|_{L^p(T^*)}\\
&\le c\left(h_{T^*}^{-2/p}\|y(u)-y_h(u)\|_{H^1(\O)}+\|y(u)-I_hy(u)\|_{L^\infty(\O)}\right)\\
&\le c\left(h_{T^*}^{-2/p}\|y(u)-I_hy(u)\|_{H^1(\O)}+\|y(u)-I_hy(u)\|_{L^\infty(\O)}\right)\\
&\le c(h^{1/2}h_{T^*}^{-2/p}+h^{1/2})|y(u)|_{H^{3/2}(\O)}\\
&\le ch^{\frac{1}{2}-\frac{2}{p\bar\mu}}|y(u)|_{H^{3/2}(\O)}
\end{align*}
with $\bar\mu:=\min_j\{\mu_j\}$.
The last estimate is due to the definition \eqref{mesh} of the mesh size of a triangle $T$. Hence, the assertion follows since $p$ can be chosen arbitrarily large.

(iii) The estimates are obtained in a standard way using the Assumptions (A\ref{A4}) and (A\ref{A6}), see also~\cite{AradaCasasTroeltzsch:2002}.

(iv) The proof of the uniform convergence of the state and the adjoint state can be found in \cite{CasasMateos:2002-2}.
\end{proof}
Now, we can prove the convergence of the discretizations. For the proof we refer to Theorem 4.4 and Theorem 4.5 in \cite{CasasMateosTroeltzsch:2005} having regard to the results of Theorem~\ref{firstorder} and Theorem~\ref{errorestimates}.
\begin{theorem}\label{uhbarconv}
 Let Assumption~\ref{Psetting} be satisfied. Moreover, let $\ub$ be a local minimum of problem (P) satisfying the second order sufficient optimality conditions given in Theorem \ref{suffcond}. Then there exist $\varepsilon>0$ and $h_0>0$ such that (P$_h$) has a local minimum $\bar u_h$ in the $L^\infty(\Gamma)$-ball around $\ub$ with radius $\varepsilon$ for every $h<h_0$. Moreover, the following convergences hold true
\begin{equation*}
\lim\limits_{h\to 0}J_h(\bar u_h)=J(\ub)\quad\text{and}\quad\lim\limits_{h\to0}\|\ub-\bar u_h\|_{L^\infty(\G)}=0. 
\end{equation*}
\end{theorem}

\section{Auxiliary estimates for the postprocessing approach}
In the sequel we denote by $\ub$ a fixed local solution of (P) satisfying the second order sufficient optimality conditions and by $\uhb$ the associated local solution of (P$_h$) converging uniformly to $\ub$. Moreover, the corresponding states and adjoint states are denoted by $\yb=y(\ub)$, $\pb=p(\yb)$ and $\yhb=y_h(\uhb)$, $\phb=p_h(\yhb)$, respectively. In our error analysis we will need a discrete control $u_h$, which is admissible for (P$_h$), close to the optimal control $\ub$ and the direction $\uhb-u_h$ should belong to the critical cone $C_\tau(\ub)$, see \eqref{criticalcone}, such that the second order sufficient condition can be applied. An intuitive choice is given by $u_h=R_h\ub$, where $R_h: C^0(\G)\rightarrow U_h$ denotes the 0-interpolator onto $U_h$ defined by:
\[
  (R_hf)(x)=f(S_E),\quad x\in E,\, E\in \mathcal E_h
\]
and $S_E$ is the midpoint of the edge $E$. The element $R_h\ub$ is indeed admissible for (P$_h$) and close to $\ub$ but $\uhb-R_h\ub$ does not necessarily belong to the critical cone. To overcome this difficulty, we modify the interpolator $R_h$. Due to the regularity of the adjoint state, see Theorem~\ref{firstorder} and Theorem~\ref{regadjoint}, and the fact that the optimal control is given by the projection formula \eqref{projform}, we can distinguish between active points $(\bar u(x)\in\{u_a,u_b\})$ and inactive points $(\bar u(x)\in(u_a,u_b))$. Based on this we can classify the edges $E\in \mathcal E_h$ in the following two sets $K_1$ and $K_2$ as in Section 2 of \cite{roeschsimon:2007}:
\begin{align*}
  K_1&:=\left\{E\in \mathcal{E}_h:\ E\ \text{contains active and inactive points}\right\},\\
  K_2&:=\left\{E\in \mathcal{E}_h:\ E\ \text{contains only active points or only inactive points}\right\}.
\end{align*}
The modified interpolation operator is now defined by
\begin{equation}\label{def_Rhu}
(R_h^{\ub} f)(x):=\left\{
\begin{array}{ll}
(R_hf)(x),&\quad\text{for }x\in E,\ E\in K_{2}\\
f(x_K),&\quad\text{for } x\in E,\ E\in K_1
\end{array}\right.
\end{equation}
with $x_K\in E$ such that either $\ub(x_K)=u_a$ or $\ub(x_K)=u_b$. We make the following assumption on the measure of the set $K_1$ which is valid in many practical applications.
\begin{assumption}\label{bound_ass}
	We suppose that $\text{\normalfont meas}(K_1)\le ch$.
\end{assumption}
\begin{remark}
Compared to linear elliptic optimal control problems the Assumption~\ref{bound_ass} is slightly stronger, cf.~\cite{MateosRoesch:2008,ApelPfeffererRoesch:2010,ApelPfeffererRoesch:2012}. In the linear case the set $K_1$ is only the union of all elements $E\in\mathcal{E}_h$ where the optimal control has kinks, whereas the present definition of the set $K_1$ also admits elements where the optimal control intersects smoothly the control constraints. However, the definition of the modified interpolation operator $R_h^{\ub}$ makes the stronger assumption necessary to prove the superconvergence properties of the postprocessed control in Section~\ref{sec:mainresults}.
\end{remark}
Now we collect approximation properties of the introduced interpolator $R_h^{\ub}$ that will be intensively used in the sequel of the paper.
\begin{lemma}\label{Rhestimates}
(i) Let $\mathcal{S}=\{1,\ldots,m\}$ and $j\in\{0\}\cup\mathcal{S}$.
For $E\in\mathcal E_{h,j}\cap K_1$ the following estimates hold true
\[
  \left|\int\limits_E(f-R_h^{\ub}f)\ds\right|\le \begin{cases}
																													ch|E||f|_{W^{1,\infty}_{1-\mu_j}(E)}&\text{if }j\in\mathcal{S},\,\mu_j\in(0,1],\,f\in W^{1,\infty}_{1-\mu_j}(E)\\
                                                          ch|E||f|_{W^{1,\infty}(E)}&\text{if }j=0,\,f\in W^{1,\infty}(E)\\
                                                       \end{cases}.
\]
For $E\in\mathcal E_{h,j}\cap K_2$ the following estimates are valid
\[
  \left|\int\limits_E(f-R_h^{\ub}f)\ds\right|\le \begin{cases}
																													ch^2|E|^{1/2}|f|_{W^{2,2}_{2(1-\mu_j)}(E)}&\text{if }j\in\mathcal{S},\,\mu_j\in(1/4,1],\,f\in W^{2,2}_{2(1-\mu_j)}(E)\\
                                                          ch^2|E|^{1/2}|f|_{W^{2,2}(E)}&\text{if }j=0,\,f\in W^{2,2}(E)
                                                       \end{cases}.
\]
(ii) Let $E\in\mathcal E_{h}$ and $f\in H^1(E)$. Then the estimate
\[
  \|f-R_h^{\bar u}f\|_{L^2(E)}\leq c h |f|_{H^1(E)}
\]
holds.
\end{lemma}
\begin{proof}
  (i) The proofs given in \cite[Section 6]{ApelPfeffererRoesch:2012} for the interpolator $R_h$ can easily be adopted to the modified interpolator~$R_h^{\ub}$ by observing that
  \begin{align}
		R_h^{\bar u}p&=p\quad\forall p\in\mathcal{P}_0(E),\,\forall E\in \mathcal{E}_h,\label{eq:deny}\\
		\int_ER_h^{\bar u}p \,\mathrm{d}s&=\int_Ep \,\mathrm{d}s\quad\forall p\in\mathcal{P}_1(E),\,\forall E\in \mathcal{E}_h\cap K_2.\notag
  \end{align}
  (ii) Based on~\eqref{eq:deny} the estimate is a direct consequence of the Deny-Lions Lemma.
\end{proof}
In the sequel the following estimates regarding the second derivative of the cost functional $J$ and its discrete counterpart $J_h$ will be useful.
\begin{lemma}\label{J''estimates}
Suppose Assumption \ref{Psetting} is satisfied. 
\begin{itemize}
 \item[(i)] Let $u\in L^2(\Gamma)$ be given. Then there holds for all $v\in L^2(\Gamma)$
\[
  |\left(J''(u)-J_h''(u)\right)[v,v]|\le ch^{1/2-\varepsilon}\|v\|_{L^2(\G)}^2
\]
with some $\varepsilon>0$.
\item[(ii)] Let $u_1,\,u_2\in L^2(\Gamma)$ be given. Then there is the estimate
\[
   |\left(J_h''(u_1)-J_h''(u_2)\right)[v,v]|\le c\|u_1-u_2\|_{L^2(\G)}\|v\|_{L^2(\G)}^2
\]
for all $v\in L^2(\Gamma)$.
\end{itemize}
\end{lemma}
\begin{proof}
(i) Based on Assumption \ref{Psetting}, several finite element error estimates, particularly \eqref{fe-error-inf}, the assertion can be obtained by straightforward calculations. For details regarding a similar problem see e.g. \cite[Lemma 6.2]{AradaCasasTroeltzsch:2002}.

(ii) Again the estimate is straightforward using the local Lipschitz continuity of the second derivatives of $d$, see also \cite[Lemma 6.3]{AradaCasasTroeltzsch:2002} for a comparable result.
\end{proof}
\begin{lemma}\label{diff_jh_estimate}
Suppose that Assumption~\ref{Psetting} is fulfilled. Then there exist $\delta'>0$ and a mesh size $h_0>0$ such that for all $h<h_0$
\begin{equation*}
 \delta'\lnorm{\uhb-R_h^{\ub}\ub}{\G}^2\le(J_h'(\uhb)-J_h'(R_h^{\ub} \ub))(\uhb-R_h^{\ub}\ub).
\end{equation*}
\end{lemma}
\begin{proof}
We proceed similar to the proof of Lemma 4.6. in~\cite{CasasMateosTroeltzsch:2005}. Let us set
\[
  \bar d(x)=\pb(x)+\nu\ub(x)\quad\text{and}\quad \bar d_h(x)=\phb(x)+\nu\uhb(x)
\]
and take $\delta$ and $\tau$ as in Theorem \ref{suffcond}. By means of Theorem \ref{uhbarconv} and Theorem \ref{errorestimates}(iv), we know that $\bar d_h$ converges uniformly to $\bar d$ on $\G$. On that account, there exists a mesh size $h_\tau$ such that
\begin{equation}\label{h_aux1}
  \|\bar d-\bar d_h\|_{L^\infty(\G)}< \frac{\tau}{4}.
\end{equation}
For every $E\in \mathcal E_h$ we define
\[
  I_E=\int\limits_{E}\bar d_h\ds.
\]
Due to the discrete variational inequality \eqref{varineqdis}, we obtain
\[
  \bar u_h|_{E}=\left\{
  \begin{aligned}
  u_a,&\text{ if }I_E>0\\
  u_b,&\text{ if }I_E<0
  \end{aligned}\right..
\]
Now, we take $0<h_1\le h_\tau$ such that $|\bar d(x_1)-\bar d(x_2)|< \frac{\tau}{4}$ if $|x_1-x_2|< h_1$. Due to \eqref{h_aux1}, we derive
\[
  \text{if }\xi\in E\text{ and }\bar d(\xi)>\tau\quad\Rightarrow\quad \bar d_h(x)>\frac{\tau}{2}\quad\forall x\in E,\,\forall E\in\mathcal E_h,\,\forall h<h_1.
\]
Thus, we have $I_E>0$ and therefore $\uhb|_{E}=u_a$. Moreover, the continuous variational inequality \eqref{varineq} implies $\ub(\xi)=u_a$. By definition \eqref{def_Rhu} of the operator $R_h^{\ub}$, we also have $(R_h^{\ub}\ub)(\xi)=u_a$. Then $(\uhb-R_h^{\ub}\ub)(\xi)=0$ whenever $\bar d(\xi)>\tau$. We obtain the analogous result for all $\xi$ with $\bar d(\xi)<-\tau$. Since $u_a\le\uhb(x)\le u_b$, one can easily see $(\uhb-R_h^{\ub}\ub)(x)\ge0$ if $\ub(x)=u_a$ and $(\uhb-R_h^{\ub}\ub)(x)\le0$ if $\ub(x)=u_b$. Thus, we proved that $\uhb-R_h^{\ub}\ub$ belongs to the $\tau$-critical cone $C_\tau(\ub)$ and \eqref{SSC} is applicable:
\begin{equation}
	\label{eq:ssc1}
  J''(\ub)[\uhb-R_h^{\ub}\ub,\uhb-R_h^{\ub}\ub]\ge \delta\lnorm{\uhb-R_h^{\ub}\ub}{\G}^2\quad\forall h<h_1
\end{equation}
Due to the mean value theorem, we obtain for some $0<\theta<1$ and $\hat u=\uhb+\theta(R_h^{\ub}\ub-\uhb)$
\[
  \begin{split}
    (J_h'(\uhb)&-J_h'(R_h^{\ub} \ub))(\uhb-R_h^{\ub}\ub)=J_h''(\hat u)[\uhb-R_h^{\ub}\ub,\uhb-R_h^{\ub}\ub]\\
    &\ge J''(\ub)[\uhb-R_h^{\ub}\ub,\uhb-R_h^{\ub}\ub]-|(J_h''(\hat u)-J_h''(\ub))[\uhb-R_h^{\ub}\ub,\uhb-R_h^{\ub}\ub]|\\
    &-|(J_h''(\ub)-J''(\ub))[\uhb-R_h^{\ub}\ub,\uhb-R_h^{\ub}\ub]|
  \end{split}
\]
By means of~\eqref{eq:ssc1} and the estimates in Lemma \ref{J''estimates} we arrive at
\[
  \begin{split}
    (J_h'(\uhb)&-J_h'(R_h^{\ub} \ub))(\uhb-R_h^{\ub}\ub)\ge \left(\delta-c\lnorm{\hat u-\ub}{\G}-ch^{1/2-\varepsilon}\right)\lnorm{\uhb-R_h^{\ub}\ub}{\G}^2
  \end{split}
\]
Thanks to the uniform convergence of $\uhb$ to $\ub$ and the approximation properties of the operator $R_h^{\ub}$, see Lemma \ref{Rhestimates}, we derive the existence of a mesh size $0<h_0\le h_1$ and $\delta'>0$ such that
\[ 
  \delta'\lnorm{\uhb-R_h^{\ub}\ub}{\G}^2\le(J_h'(\uhb)-J_h'(R_h^{\ub} \ub))(\uhb-R_h^{\ub}\ub)
\]
is valid for all $h<h_0$.
\end{proof}
In the following we introduce the $L^2$-projection into the space of piecewise constant functions on the boundary. Furthermore, we will state some useful properties. It is well-known that the $L^2$-projection $Q_hf$ of a function $f\in L^2(\G)$ into the space $U_h$ satisfies
\[
  Q_hf\equiv\frac{1}{|E|}\int\limits_Ef(x)\ds\quad\quad\text{for every element }E\in\mathcal E_h.
\]
The following approximation property of $Q_h$ is proven in Corollary 4.8 of \cite{ApelPfeffererRoesch:2010}.
\begin{lemma}\label{L2proj-scalar}
 For any element $E\in\mathcal E_h$ and any function $f\in H^1(E)$ and $v\in H^s(E)$, $s\in[0,1]$, the estimate
\[
  \scalar{f-Q_hf}{v}{E}\le c h_E^{s+1}|f|_{H^1(E)}|v|_{H^s(E)}
\]
is valid.
\end{lemma}
\begin{lemma}\label{yhu-yhrhu}
Suppose that Assumptions \ref{Psetting} and \ref{bound_ass} hold. Then there exists a mesh size $h_0>0$ such that for all $h<h_0$ and mesh parameters $\vec1/2<\vec\mu<\vec1/4+\vec\lambda/2$ the estimate
\[
  \lnorm{y_h(\ub)-y_h(R_h^{\ub}\ub)}{\O}\le ch^2
\]
holds true.
\end{lemma}
\begin{proof}
First note, that $y_h(\ub)$ and $y_h(R_h^{\ub}\ub)$ are the solutions of the discrete state equation \eqref{vardisPDE} with respect to the right hand sides $\ub$ and $R_h^{\ub}\ub$, respectively. Initially we introduce a dual problem and its discrete counterpart following the ideas of \cite[Appendix]{CasasMateos:2008}: let $\phi\in H^1(\O)$ be the unique solution of
\[
  a(\phi,v)+\into{\alpha\phi v}=\into{(y_h(\ub)-y_h(R_h^{\ub}\ub))v}\quad\quad\forall v\in H^1(\O),
\]
with
\[
  \alpha(x)=\left\{
  \begin{aligned}
  \frac{d(x,y_h(\ub)(x))-d(x,y_h(R_h^{\ub}\ub)(x))}{y_h(\ub)(x)-y_h(R_h^{\ub}\ub)(x)},&\quad\text{if }y_h(\ub)(x)-y_h(R_h^{\ub}\ub)(x)\not=0\\
  0,&\quad\text{otherwise}
  \end{aligned}
  \right..
\]
Due to the approximation properties of $R_h^{\bar u}$ and the monotonicity of $d$ according to Assumption~\ref{Psetting} one can easily check that there is a $h_0$ such that the problem is well-posed for all $h<h_0$. The corresponding discrete counterpart $\phi_h\in V_h$ is the unique solution of the problem
\[
  a(\phi_h,v_h)+\into{\alpha\phi_h v_h}=\into{(y_h(\ub)-y_h(R_h^{\ub}\ub))v_h}\quad\forall v_h\in V_h.
\]
By means of $y_h(\ub), y_h(R_h^{\ub}\ub)\in V_h$ as solutions of \eqref{vardisPDE} and the definition of $\alpha$, we derive
\begin{align}
  \lnorm{y_h(\ub)-y_h&(R_h^{\ub}\ub)}{\O}^2=a(\phi_h, y_h(\ub)-y_h(R_h^{\ub}\ub))+\into{\alpha\phi_h(y_h(\ub)- y_h(R_h^{\ub}\ub))}\notag\\
  &=a(y_h(\ub)-y_h(R_h^{\ub}\ub),\phi_h)+\into{(d(x,y_h(\ub))-d(x,y_h(R_h^{\ub}\ub))\phi_h}\notag\\
  &=\intg{(\ub-R_h^{\ub}\ub)\phi_h}.\label{eq:u-rhu01}
\end{align}
Next,  we split the last term in two terms and estimate them separately:
\begin{equation}\label{u-rhu1}
 \scalar{\ub-R_h^{\ub}\ub}{\phi_h}{\G}= \scalar{\ub-R_h^{\ub}\ub}{\phi_h-\phi}{\G}+ \scalar{\ub-R_h^{\ub}\ub}{\phi}{\G}.
\end{equation}
For the first term, we derive
\begin{equation}\label{u-rhu2}
 \begin{aligned}
   \scalar{\ub-R_h^{\ub}\ub}{\phi_h-\phi}{\G}&\le c\lnorm{\ub-R_h^{\ub}\ub}{\G}\|\phi_h-\phi\|_{H^1(\O)}\\
  &\le ch^2|\ub|_{H^1(\G)}\lnorm{y_h(\ub)-y_h(R_h^{\ub}\ub)}{\O}
 \end{aligned}
\end{equation}
using the Cauchy-Schwarz inequality, a standard trace theorem, Lemma~\ref{Rhestimates} and Lemma \ref{feerrorgenPDE}. The second term in \eqref{u-rhu1} is again split into two terms:
\[
  \scalar{\ub-R_h^{\ub}\ub}{\phi}{\G}= \scalar{\ub-Q_h\ub}{\phi}{\G}+\scalar{Q_h\ub-R_h^{\ub}\ub}{\phi}{\G}.
\]
According to Lemma~\ref{pro:reglinearPDE} the solution $\phi$ of the previously introduced dual problem belongs to $\WbO{2}$ for $\vec\beta$ satisfying~\eqref{condbeta} since $y_h(\ub)-y_h(R_h^{\ub}\ub)\in L^2(\O)\hookrightarrow W^{0,2}_{\vec \beta}(\O)$. Furthermore, the a priori estimate
\begin{equation}
\label{eq:aprioriphi}
\|\phi\|_{W^{2,2}_{\vec\beta}(\O)}\leq c\|y_h(\ub)-y_h(R_h^{\ub}\ub)\|_{L^2(\Omega)}
\end{equation}
is valid. Thus, Lemma~\ref{L2proj-scalar}, the trace theorem and embeddings in classical and weighted Sobolev spaces yield
\begin{equation}\label{u-rhu3}
  \begin{aligned}
 \scalar{\ub-Q_h\ub}{\phi}{\G}&\le ch^2|\ub|_{H^1(\G)}|\phi|_{H^1(\G)}\\
  &\le ch^2|\ub|_{H^1(\G)}\lnorm{y_h(\ub)-y_h(R_h^{\ub}\ub)}{\O},
  \end{aligned}
\end{equation}
cf. the proof of Theorem~\ref{firstorder} or the proof of Lemma 7.4. in~\cite{ApelPfeffererRoesch:2012} for details. We proceed with
\[
  \begin{aligned}
  \scalar{Q_h\ub-R_h^{\ub}\ub}{\phi}{\G}&\le \|Q_h\ub-R_h^{\ub}\ub\|_{L^1(\G)}\|\phi\|_{L^\infty(\G)}\\
  &\le c  \|Q_h\ub-R_h^{\ub}\ub\|_{L^1(\G)}\lnorm{y_h(\ub)-y_h(R_h^{\ub}\ub)}{\O}
  \end{aligned}
\]
applying again embeddings in classical and weighted Sobolev spaces and~\eqref{eq:aprioriphi}. Since $R_h^{\ub}\ub$ is constant on every element $E$, we derive
\begin{align*}
	\|Q_h\bar u &-R_h^{\ub} \bar u\|_{L^1(\Gamma)}=\|Q_h\left(\bar u -R_h^{\ub} \bar u\right)\|_{L^1(\Gamma)}= \sum_{E\in \mathcal{E}_h}\left|\int_E \left(\bar u - R_h^{\ub} \bar u \right)ds\right|\\
	&=\sum_{j=0}^m\left(\sum_{E\in\mathcal{E}_{h,j}\cap{K_1}}\left|\int_E \left(\bar u - R_h^{\ub} \bar u \right)ds\right|+\sum_{E\in\mathcal{E}_{h,j}\cap {K_2}}\left|\int_E \left(\bar u - R_h^{\ub} \bar u \right)ds\right|\right).
\end{align*}
By means of Lemma \ref{Rhestimates} we obtain
\begin{align}
\|Q_h\ub-R_h^{\ub}\ub&\|_{L^1(\G)}\leq c\left(\sum_{E\in\mathcal{E}_{h,0}\cap {K_1}}h |E||
	  \bar u|_{W^{1,\infty}(E)}+\sum_{j=1}^m\sum_{E\in\mathcal{E}_{h,j}\cap {K_1}}h |E||\bar u|_{W^{1,\infty}_{1-\mu_j}(E)}\right.\notag\\
	  &+\left.\sum_{E\in\mathcal{E}_{h,0}\cap {K_2}}h^2 |E|^{1/2}|
	  \bar u|_{W^{2,2}(E)}+\sum_{j=1}^m\sum_{E\in\mathcal{E}_{h,j}\cap {K_2}}h^2 |E|^{1/2}|\bar u|_{W^{2,2}_{2(1-\mu_j)}(E)}\right)\notag\\
	  &\leq ch |K_1| \left(|\bar u|_{W^{1,\infty}(K_1\cap\Gamma^0)}+\sum_{j=1}^m|\bar u|_{W^{1,\infty}_{1-\mu_j}(K_1\cap\Gamma_j^\pm)}\right)\nonumber\\
	  &+ch^2 |K_2|^{1/2}\left( |\bar u|_{W^{2,2}(K_2\cap\Gamma^0)}+\sum_{j=1}^m|\bar u|_{W^{2,2}_{2(1-\mu_j)}(K_2\cap\Gamma^\pm_j)}\right)\nonumber\\
	  &\leq ch^2\left(|\bar u |_{W^{1,\infty}_{\vec 1-\vec \mu}(K_1)}+|\bar  u|_{W^{2,2}_{2(\vec 1 -\vec \mu)}(K_2)}\right),\label{qhu-rhu}
\end{align}
where we applied the discrete Cauchy-Schwarz inequality and Assumption \ref{bound_ass}. Hence, we end up with
\begin{equation}\label{u-rhu4}
  \scalar{Q_h\ub-R_h^{\ub}\ub}{\phi}{\G}\le ch^2(|\ub|_{W^{1,\infty}_{1-\mu}(K_1)}+|\ub|_{W^{2,2}_{2(1-\mu)}(K_2)})\lnorm{y_h(\ub)-y_h(R_h^{\ub}\ub)}{\O}.
\end{equation}
Collecting the intermediate estimates from \eqref{eq:u-rhu01}-\eqref{u-rhu4}, we derive
\begin{equation*}
  \lnorm{y_h(\ub)-y_h(R_h^{\ub}\ub)}{\O}\le ch^2(|\ub|_{H^1(\G)}+|\ub|_{W^{1,\infty}_{1-\mu}(K_1)}+|\ub|_{W^{2,2}_{2(1-\mu)}(K_2)})
\end{equation*}
According to Theorem \ref{firstorder} the optimal control $\ub$ is bounded in the space $H^1(\G)$. Taking into account that the optimal control $\ub$ is given by the projection formula~\eqref{projform}, we can split the boundary $\Gamma$ in an ``active'' part $\mathcal A$ ($\bar u=u_a$ or $\bar u=u_b$) and an ``inactive part'' $\mathcal I$ ($\bar u =-\bar p/\nu$). Hence, we can estimate
	\begin{align*}
	  |\bar u |_{W^{1,\infty}_{\vec 1-\vec \mu}(K_1)}+|\bar  u|_{W^{2,2}_{2(\vec 1 -\vec \mu)}(K_2)}=
	  |\bar p |_{W^{1,\infty}_{\vec 1-\vec \mu}(K_1\cap\mathcal{I})}+|\bar  p|_{W^{2,2}_{2(\vec 1 -\vec \mu)}(K_2\cap\mathcal{I})}&\leq c
	\end{align*}
	by using Theorem~\ref{regadjoint} with $\vec 1/2<\vec\mu<\vec 1/4+\vec\lambda/2$. This ends the proof.
\end{proof}
We will continue with a supercloseness result for $\lnorm{\uhb-R_h^{\ub}\ub}{\G}$.
\begin{lemma}\label{lem_aux_uh-rhu}
Let Assumption \ref{Psetting} and \ref{bound_ass} be satisfied. Then there exists a mesh size $h_0>0$ such that for all $h<h_0$ the estimate
\begin{equation*}
  \lnorm{\uhb-R_h^{\ub}\ub}{\G}\le ch^{3/2}
\end{equation*}
is valid for mesh grading parameters $\vec1/2<\vec\mu<\vec 1/4+\vec\lambda/2$.
\end{lemma}
\begin{proof}
We start with the pointwise a.e. version of the variational inequality \eqref{varineq}:
\[
  (\pb(x)+\nu\ub(x))\cdot(u-\ub(x))\ge0\quad\quad \forall u\in[u_a,u_b].
\]
We apply this formula for $x=S_E$, $E\in K_2$ and $u=\uhb(S_E)$ and arrive at
\[
  (\pb(S_E)+\nu\ub(S_E))\cdot(\uhb(S_E)-\ub(S_E))\ge0\quad\quad \forall S_E,\,E\in K_2.
\]
Analogously, we can apply this formula to elements $E$ of the subset $K_1$ using $x_K$ instead of the midpoint $S_E$, where $x_K\in E$ is a point satisfying either $\ub(x_K)=u_a$ or $\ub(x_K)=u_b$. Integrating these formulas over $E$, summing up over all $E\in \mathcal E_h$ and taking into account the definition of $R^{\ub}_h$ in \eqref{def_Rhu}, we find
\[
  \scalar{R_h^{\ub}\pb+\nu R_h^{\ub}\ub}{\uhb-R_h^{\ub}\ub}{\G}\ge 0.
\]
Next we test the discrete variational inequality \eqref{varineqdis} with the function $R_h^{\ub}\ub\in U_{ad,h}$ and get
\[
 \scalar{\phb+\nu \uhb}{R_h^{\ub}\ub-\uhb}{\G}\ge 0.
\]
Adding the last two inequalities and inserting appropriate intermediate functions yields
\begin{equation}\label{uh-rhu0}
  \begin{aligned}
  0&\le \scalar{R_h^{\ub}\pb-\bar p_h+\nu (R_h^{\ub}\ub-\bar u_h)}{\uhb-R_h^{\ub}\ub}{\G}\\
  &=\scalar{R_h^{\ub}\pb-\pb}{\uhb-R_h^{\ub}\ub}{\G}+\scalar{\pb-p_h(y_h(R_h^{\ub}\ub))}{\uhb-R_h^{\ub}\ub}{\G}\\
  &+\scalar{p_h(y_h(R_h^{\ub}\ub))-\phb+\nu (R_h^{\ub}\ub-\uhb)}{\uhb-R_h^{\ub}\ub}{\G}
  \end{aligned}
\end{equation}
Note, that $p_h(y_h(R_h^{\ub}\ub))$ denotes the solution of the discrete adjoint equation \eqref{disadjoint} w.r.t. the state $y_h(R_h^{\ub}\ub)$. The last term in the previous estimate can be formulated as
\[
  (J_h'(R_h^{\ub}\ub)-J_h'(\uhb))(\uhb-R_h^{\ub}\ub)
\]
such that Lemma \ref{diff_jh_estimate} can be applied and we obtain that there is a $h_1>0$ such that
\begin{equation}\label{uh-rhu1}
  \begin{aligned}
 \delta'\lnorm{\uhb-R_h^{\ub}\ub}{\G}^2&\le \scalar{R_h^{\ub}\pb-\pb}{\uhb-R_h^{\ub}\ub}{\G}+\scalar{\pb-p_h(y_h(R_h^{\ub}\ub))}{\uhb-R_h^{\ub}\ub}{\G}
  \end{aligned}
\end{equation}
for all $h<h_1$. The first term can be written as
\[
 \begin{aligned}
    \scalar{R_h^{\ub}\pb-\pb}{\uhb-R_h^{\ub}\ub}{\G}&=\sum\limits_{E\in  K_1}(\uhb-R_h^{\ub}\ub)|_E\int_E(R_h^{\ub}\pb-\pb) \ds\\
    &+\sum\limits_{E\in  K_2}(\uhb-R_h^{\ub}\ub)|_E\int_E(R_h^{\ub}\pb-\pb) \ds.
 \end{aligned}
\]
Adapting the estimates of \eqref{qhu-rhu} and using the formula 
\[
 \lnorm{\uhb-R_h^{\ub}\ub}{E}=|E|^{1/2}|(\uhb-R_h^{\ub}\ub)|_E|,
\]
we obtain
\[
 \begin{aligned}
    \scalar{R_h^{\ub}\pb-\pb}{\uhb-R_h^{\ub}\ub}{\G}&\le ch|K_1|^{1/2} \lnorm{\uhb-R_h^{\ub}\ub}{K_1}|\pb|_{W^{1,\infty}_{\vec1-\vec\mu}(K_1)}\\
  &+ch^2\lnorm{\uhb-R_h^{\ub}\ub}{K_2}|\pb|_{W^{2,2}_{2(\vec1-\vec\mu)}(K_2)}.
  \end{aligned}
\]
Due to Assumption \ref{bound_ass} and Theorem~\ref{regadjoint} we derive for $\vec 1/2<\vec \mu<\vec 1/4+\vec\lambda/2$
\[
  \scalar{R_h^{\ub}\pb-\pb}{\uhb-R_h^{\ub}\ub}{\G}\le ch^{3/2}\lnorm{\uhb-R_h^{\ub}\ub}{\G}
\]
For the second term in \eqref{uh-rhu1} we first find
\[
 \scalar{\pb-p_h(y_h(R_h^{\ub}\ub))}{\uhb-R_h^{\ub}\ub}{\G}\le \lnorm{\pb-p_h(y_h(R_h^{\ub}\ub))}{\G}\lnorm{\uhb-R_h^{\ub}\ub}{\G}
\]
such that we continue by applying the triangle inequality
\begin{equation}\label{pu-phrhu}
 \begin{aligned}
 \lnorm{\pb-p_h(y_h(R_h^{\ub}\ub))}{\G}&\le \lnorm{\pb-p_h(\yb)}{\G}+\lnorm{p_h(\yb)-p_h(y_h(\ub))}{\G}\\
  & +\lnorm{p_h(y_h(\ub))-p_h(y_h(R_h^{\ub}\ub))}{\G},
 \end{aligned}
\end{equation}
where $p_h(\yb)$ and $p_h(y_h(\ub))$ denote the solution of the discrete adjoint state equation \eqref{disadjoint} w.r.t. the states $\yb$ and $y_h(\ub)$, respectively.
Thus the first error on the right side in \eqref{pu-phrhu} is a finite element error on the boundary for the adjoint states and we apply Lemma \ref{feerrorgenPDE}(ii) and Theorem~\ref{regadjoint} such that
\[
   \lnorm{\pb-p_h(\yb)}{\G}\le ch^2|\ln h|^{3/2}
\]
for $\vec 1/4<\vec\mu<\vec1/4+\vec \lambda/2$. For the second and the third term in \eqref{pu-phrhu} one can prove by a standard trace theorem, Theorem~\ref{errorestimates}(iii) and (i), Theorem~\ref{firstorder}, and Lemma~\ref{yhu-yhrhu} with $h<h_0\leq h_1$ that
\[
 \begin{split}
 \lnorm{p_h(\yb)-p_h(y_h(\ub))}{\G}+&\lnorm{p_h(y_h(\ub))-p_h(y_h(R_h^{\ub}\ub))}{\G}\\
  &\le c(\lnorm{\yb-y_h(\ub)}{\O}+\lnorm{y_h(\ub)-y_h(R_h^{\ub}\ub)}{\O})\\
  &\le c h^2
  \end{split}
\]
provided that $\vec 1/2<\vec\mu<\vec1/4+\vec \lambda/2$. Summarizing the previous estimates yields the assertion.
\end{proof}
We announced in the introduction of this paper, that we want to carry over the results for linear quadratic problems to optimal control problems governed by semilinear equations. Unfortunately, the supercloseness result derived in the previous lemma cannot be improved. The reason is that on the set $K_1$ the integration formula
\[
  \int_E(R_h^{\ub}f-f)\ds=0,\quad\quad E\in K_1
\]
induced by our modified interpolator $R_h^{\ub}$ is in general only exact for constant polynomials $f$ on the element $E$, since the interpolation point is in general not the midpoint $S_E$ of the element $E$. This is different to the linear quadratic case considered in \cite{ApelPfeffererRoesch:2012}. However, if we restrict to the set $K_2$, we can improve the estimate.
\begin{lemma}\label{lem_aux_uh-rhu_K2}
Let Assumption~\ref{Psetting} and Assumption \ref{bound_ass} be satisfied. Then there exists a mesh size $h_0>0$ such that for all $h<h_0$ the estimate
\begin{equation*}
  \lnorm{\uhb-R_h^{\ub}\ub}{K_2}\le ch^2|\ln h|^{3/2}
\end{equation*}
is valid for mesh grading parameters $\vec1/2<\vec\mu<\vec1/4+\vec\lambda/2$. 
\end{lemma}
\begin{proof}
We can follow the lines of the proof of Lemma \ref{lem_aux_uh-rhu} up to formula \eqref{uh-rhu0} considering only the set $K_2$, i.e., we find
\begin{equation}\label{uh-rhuk_0}
 \begin{aligned}
  0&\le \scalar{R_h^{\ub}\pb-\bar p_h+\nu (R_h^{\ub}\ub-\bar u_h)}{\uhb-R_h^{\ub}\ub}{K_2}\\
  &=\scalar{R_h^{\ub}\pb-\pb}{\uhb-R_h^{\ub}\ub}{K_2}+\scalar{\pb-p_h(y_h(R_h^{\ub}\ub))}{\uhb-R_h^{\ub}\ub}{K_2}\\
  &+\scalar{p_h(y_h(R_h^{\ub}\ub))-\phb+\nu (R_h^{\ub}\ub-\uhb)}{\uhb-R_h^{\ub}\ub}{K_2}.
  \end{aligned}
\end{equation}
Introducing the characteristic functions $\chi_{K_i}$ w.r.t. to the set $K_i$, $i=1,2$, the last scalar product in the previous formula can be interpreted as
\begin{equation}\label{uh-rhuk_1}
  (J_h'(R_h^{\ub}\ub)-J_h'(\uhb))(\chi_{K_2}(\uhb-R_h^{\ub}\ub)).
\end{equation}
Unfortunately, the result of Lemma \ref{diff_jh_estimate} is not directly applicable. Analogously to the proof of Lemma \ref{diff_jh_estimate}, we obtain for some $0<\theta<1$ and $\hat u=\uhb+\theta(R_h^{\ub}\ub-\uhb)$
\begin{equation}\label{uh-rhuk_2}
   (J_h'(\uhb)-J_h'(R_h\ub))(\chi_{K_2}(\uhb-R_h^{\ub}\ub))=J_h''(\hat u)[\uhb-R_h^{\ub}\ub,\chi_{K_2}(\uhb-R_h^{\ub}\ub)]
\end{equation}
due to the mean value theorem. We continue by inserting appropriate intermediate terms
\[
  \begin{split}
    J_h&''(\hat u)[\uhb-R_h^{\ub}\ub,\chi_{K_2}(\uhb-R_h^{\ub}\ub)]\ge J''(\ub)[\chi_{K_2}(\uhb-R_h^{\ub}\ub),\chi_{K_2}(\uhb-R_h^{\ub}\ub)]\\
    &-|(J_h''(\ub)-J''(\ub))[\chi_{K_2}(\uhb-R_h^{\ub}\ub),\chi_{K_2}(\uhb-R_h^{\ub}\ub)]|\\
    &-|(J_h''(\hat u)-J_h''(\ub))[\chi_{K_2}(\uhb-R_h^{\ub}\ub),\chi_{K_2}(\uhb-R_h^{\ub}\ub)]|\\
    &-|J_h''(\hat u)[\uhb-R_h^{\ub}\ub,\chi_{K_2}(\uhb-R_h^{\ub}\ub)]-J_h''(\hat u)[\chi_{K_2}(\uhb-R_h^{\ub}\ub),\chi_{K_2}(\uhb-R_h^{\ub}\ub)]|.
  \end{split}
\]
The first three addends can be estimated by means of the second order sufficient optimality conditions and further estimates regarding second derivatives of the cost functional as in the proof of Lemma \ref{diff_jh_estimate}. Thus, there is a constant $\delta'>0$ and a mesh size $h_0>0$
such that for all $h<h_0$
\begin{equation}\label{uh-rhuk_3}
  \begin{split}
    J_h''(\hat u)[\uhb-R_h^{\ub}\ub,\chi_{K_2}(\uhb-R_h^{\ub}\ub)]&\ge \delta'\lnorm{\uhb-R_h^{\ub}\ub}{K_2}^2-|J_h''(\hat u)[v_1,v_2]|\\
  \end{split}
\end{equation}
where we introduced the abbreviations $v_1:=\chi_{K_1}(\uhb-R_h^{\ub}\ub)$ and $v_2:=\chi_{K_2}(\uhb-R_h^{\ub}\ub)$. We obtain for the last term in the previous estimate
\begin{equation}\label{uh-rhuk_4}
 \begin{split}
|J_h''(\hat u)[v_1,v_2]|=\left|\into{y_h^{v_1}y_h^{v_2}-p_h(y_h(\hat u))d_{yy}(x,y_h(\hat u))y_h^{v_1}y_h^{v_2}}+\nu\intg{v_1v_2}\right|,
 \end{split}
\end{equation}
where the last term vanishes by construction. We continue by
\begin{equation}\label{uh-rhuk_5}
\begin{split}
\left|\into{y_h^{v_1}y_h^{v_2}-p_h(y_h(\hat u))d_{yy}(x,y_h(\hat u))y_h^{v_1}y_h^{v_2}}\right|&\le c\lnorm{y_h^{v_1}}{\O}\lnorm{y_h^{v_2}}{\O}\\
&\le c\|v_1\|_{L^1(\Gamma)}\|v_2\|_{L^1(\Gamma)}
\end{split}
\end{equation}
due to the Cauchy-Schwarz inequality, the uniform boundedness of the discrete variables $p_h(y_h(\hat u))$ and $y_h(\hat u)$, and Lemma~\ref{l1estimate_linearizedstates}. Combining \eqref{uh-rhuk_0}-\eqref{uh-rhuk_5}, we derive
\begin{equation*}
 \begin{aligned}
  \delta'\lnorm{\uhb-R_h^{\ub}\ub}{K_2}^2&\le\scalar{R_h^{\ub}\pb-\pb}{\uhb-R_h^{\ub}\ub}{K_2} \\
  &+\scalar{\pb-p_h(y_h(R_h^{\ub}\ub))}{\uhb-R_h^{\ub}\ub}{K_2}\\
  &+c\|\uhb-R_h^{\ub}\ub\|_{L^1(K_1)}\|\uhb-R_h^{\ub}\ub\|_{L^1(K_2)}.
  \end{aligned}
\end{equation*}
Now, the first term can be estimated by
\[
  \scalar{R_h^{\ub}\pb-\pb}{\uhb-R_h^{\ub}\ub}{K_2}\le ch^2\lnorm{\uhb-R_h^{\ub}\ub}{K_2},
\]
following the lines of the proof of the previous lemma. Furthermore, the second term was already estimated in a similar form in the proof of Lemma \ref{lem_aux_uh-rhu} such that
\[
\scalar{\pb-p_h(y_h(R_h^{\ub}\ub))}{\uhb-R_h^{\ub}\ub}{K_2} \le ch^2|\ln h|^{3/2}\lnorm{\uhb-R_h^{\ub}\ub}{K_2},
\]
provided that the mesh grading parameters satisfy $\vec1/2<\vec\mu<\vec1/4+\vec\lambda/2$. Applying the H\"older inequality, Assumption \ref{bound_ass} and Lemma \ref{lem_aux_uh-rhu}, we further obtain
\[
  \|\uhb-R_h^{\ub}\ub\|_{L^1(K_1)}\le |K_1|^{1/2}\lnorm{\uhb-R_h^{\ub}\ub}{K_1}\le ch^2.
\]
Summarizing, we can prove the assertion
\[
  \lnorm{\uhb-R_h^{\ub}\ub}{K_2}\le ch^2|\ln h|^{3/2}
\]
if the mesh grading parameters satisfy $\vec1/2<\vec\mu<\vec1/4+\vec\lambda/2$.
\end{proof}
\section{Main result}\label{sec:mainresults}
This section is concerned with the error estimates for the postprocessing approach. As introduced in Section~\ref{sec:discretization} the control is approximated by piecewise constant functions. Afterwards the control $\tilde u_h$ is calculated by a projection of the discrete adjoint state $\bar p_h$ to the admissible set $U_{ad}$:
\begin{equation*}
  \tilde u_h:=\Pi_{[u_a,u_b]}\left(-\frac{1}{\nu}\bar p_h\right).
\end{equation*}
This projection is piecewise linear and continuous, but the constructed control does not belong to the discrete admissible set in general. However, we will prove that $\tilde u_h$ possesses superconvergence properties.
\begin{theorem}\label{mainresult}
Suppose that Assumption \ref{Psetting} and Assumption \ref{bound_ass} are fulfilled. Then there exists a mesh size $h_0>0$ such that for all $h<h_0$ the estimate
\[
\lnorm{\yb-\yhb}{\O}+\lnorm{\pb-\phb}{\G}+\lnorm{\ub-\tilde u_h}{\G}\le ch^2|\ln h|^{3/2}
\]
is valid provided that $\vec1/2<\vec\mu<\vec1/4+\vec\lambda/2$.
\end{theorem}
\begin{proof}
We introduce intermediate functions and apply the triangle inequality such that
\[
  \begin{aligned}
   \lnorm{\yb-\yhb}{\O}&\le \lnorm{\yb-y_h(\ub)}{\O}+\lnorm{y_h(\ub)-y_h(R_h^{\ub}\ub)}{\O}+\lnorm{y_h(R_h^{\ub}\ub)-\yhb}{\O}.
  \end{aligned}
\]
The first error term is a usual finite element error for semilinear elliptic PDEs and we rely on results given in Theorem \ref{errorestimates}(i) and Theorem~\ref{firstorder}. The second term was estimated separately in Lemma \ref{yhu-yhrhu} for all $h<h_0$. By means of Lemma \ref{l1estimate_states} the third term is estimated as follows
\[
  \begin{aligned}
  \lnorm{y_h(R_h^{\ub}\ub)-\yhb}{\O}&\le c(\|R_h^{\ub}\ub-\uhb\|_{L^1(K_1)}+\|R_h^{\ub}\ub-\uhb\|_{L^1(K_2)})\\
  &\le c(|K_1|^{1/2}\lnorm{R_h^{\ub}\ub-\uhb}{K_1}+\lnorm{R_h^{\ub}\ub-\uhb}{K_2})\\
  &\le c(h^2+h^2|\ln h|^{3/2}),
  \end{aligned}
\]
using Assumption \ref{bound_ass} and the results derived in the Lemmata \ref{lem_aux_uh-rhu} and \ref{lem_aux_uh-rhu_K2}, respectively. Thus we can conclude
\[
  \lnorm{\yb-\yhb}{\O}\le ch^2|\ln h|^{3/2}
\]
provided that the mesh grading parameters satisfy $\vec1/2<\vec\mu<\vec1/4+\vec\lambda/2$. For the error in the adjoint states we obtain
\[
  \lnorm{\pb-\phb}{\G}\le  \lnorm{\pb-p_h(\yb)}{\G}+\lnorm{p_h(\yb)-\phb}{\G}
\]
introducing the intermediate adjoint state $p_h(\yb)$ as the solution of the discrete adjoint state equation \eqref{disadjoint} w.r.t. the state $\yb$. Hence, the first error is a finite element error on the boundary for the adjoint states and  we apply Lemma \ref{feerrorgenPDE}(ii) and Theorem~\ref{firstorder} such that
\[
  \lnorm{\pb-p_h(\yb)}{\G}\le ch^2|\ln h|^{3/2}
\]
for $\vec1/2<\vec\mu<\vec1/4+\vec\lambda/2$. Theorem~\ref{errorestimates}(iii) yields the following estimate for the second term
\[
  \lnorm{p_h(\yb)-\phb}{\G}\le c\lnorm{\yb-\yhb}{\O}.
\]
Thus, the proven error estimate for the state gives the overall error estimate
\[
    \lnorm{\bar p-\phb}{\G}\le ch^2|\ln h|^{3/2}
\]
for $\vec1/2<\vec\mu<\vec1/4+\vec\lambda/2$. Since the projection operator is Lipschitz continuous we obtain
\[
 \begin{aligned}
    \lnorm{\ub-\tilde u_h}{\G}&=\left\|\Pi_{[u_a,u_b]}\left(-\frac{1}{\nu}\pb\right)-\Pi_{[u_a,u_b]}\left(-\frac{1}{\nu}\phb\right)\right\|_{L^2(\G)}\\
    &\le c \lnorm{\pb-\phb}{\G}\le ch^2|\ln h|^{3/2},
 \end{aligned}
\]
where we used the error estimate for the adjoint states proven in the step before for $\vec1/2<\vec\mu<\vec1/4+\vec\lambda/2$.
\end{proof}
\begin{remark}
  The same convergence rates can also be proven for the concept of variational discretizations if one takes into account the improved finite element error estimates on the boundary. This concept was first introduced in~\cite{Hinze:2005} for linear elliptic control problems with distributed control and in~\cite{CasasMateos:2008} for semilinear Neumann boundary control problems.
\end{remark}
\begin{remark}
  A convergence order of two can analogously be proven for distributed control problems. In that case the condition $\vec \mu<\vec\lambda$ is sufficient since one only needs error estimates in the domain.
\end{remark}
\section{Numerical example}
In this section we present a numerical example that illustrates the proven error estimates of the previous section. The example is a slightly modified version of the one presented in \cite{MateosRoesch:2008}. Let $r,\,\varphi$ be the polar coordinates located at the origin. For $\omega\in(0,2\pi)$ we define the circular sector $S_\omega:=\{x\in\mathbb{R}^2:\,(r(x),\varphi(x))\in(0,\sqrt{2}]\times [0,\omega]\}$. Moreover, let $\Omega_\omega=(-1,1)^2\cap S_\omega$ with the boundary $\G_\omega$. We are interested in the optimal control problems (QP)
\begin{gather*}
\text{min}\quad F_\omega(y,u):=\frac{1}{2}\int\limits_{\O_\omega}(y-y_d)^2\,dx+\frac{1}{2}\int\limits_{\G_\omega}u^2\ds+\int\limits_{\G_\omega}g_2y\ds\\
\begin{aligned}
-\Delta y+y+y^3&=f&&\quad\text{in }\Omega_\omega\\
\partial_n y&=u+g_1&&\quad\text{on }\Gamma_\omega
\end{aligned}\\
u_a\leq u\leq u_b\quad \text{a.e. on }\G_\omega.
\end{gather*}
Again we set $\lambda=\pi/\omega$. Let us define the functions
\begin{align*}
f(x)&=r^\lambda(x)\cos(\lambda\varphi(x))+r^{3\lambda}(x)\cos^3(\lambda\varphi(x))&&\quad\text{in }\O_\omega,\\
y_d(x)&=2r^\lambda(x)\cos(\lambda\varphi(x)+3r^{3\lambda}(x)\cos^3(\lambda\varphi(x))&&\quad\text{in }\O_\omega
\end{align*}
and
\begin{align*}
 g_1(x)&=\partial_n(r^\lambda(x)\cos(\lambda\varphi(x)))-\Pi_{[u_a,u_b]}(r^\lambda(x)\cos(\lambda\varphi(x)))&&\quad\text{on }\G_\omega,\\
 g_2(x)&=-\partial_n(r^\lambda(x)\cos(\lambda\varphi(x)))&&\quad\text{on }\G_\omega.
\end{align*}
Furthermore, let us set $u_a=-0.8$ and $u_b=0.8$. One can easily check, that
\begin{align*}
\yb(x)&=r^\lambda(x)\cos(\lambda\varphi(x)),\\
\pb(x)&=-r^\lambda(x)\cos(\lambda\varphi(x)),\\
\ub(x)&=\Pi_{[u_a,u_b]}\left(r^\lambda(x)\cos(\lambda\varphi(x))\right)
\end{align*}
satisfy the respective first order optimality conditions. Moreover, the second order sufficient optimality condition \eqref{SSC} is fulfilled by construction. The functions $\yb$, $\pb$ and $\ub$ have exactly the singular behavior discussed in Theorem~\ref{firstorder}. For the solution of the optimal control problems (QP), a standard SQP method was implemented, see e.g. Heinkenschloss and Tr\"oltzsch \cite{HeinkenschlossTroeltzsch:1999}, Kelley and Sachs \cite{KelleySachs:1990} or Kunisch and Sachs \cite{KunischSachs:1992}. The resulting quadratic subproblems were solved by applying a primal dual active set strategy according to Bergonioux, Ito and Kunisch \cite{BergouniouxItoKunisch:1999}. We also refer to Kunisch and R\"osch \cite{KunischRoesch:2002}. The discrete solutions of the PDEs have been computed using a finite element method on graded meshes as introduced in the beginning of Section 3, see Figure~\ref{gradedmesh}.
\begin{figure}
\begin{center}
\includegraphics[scale=0.35]{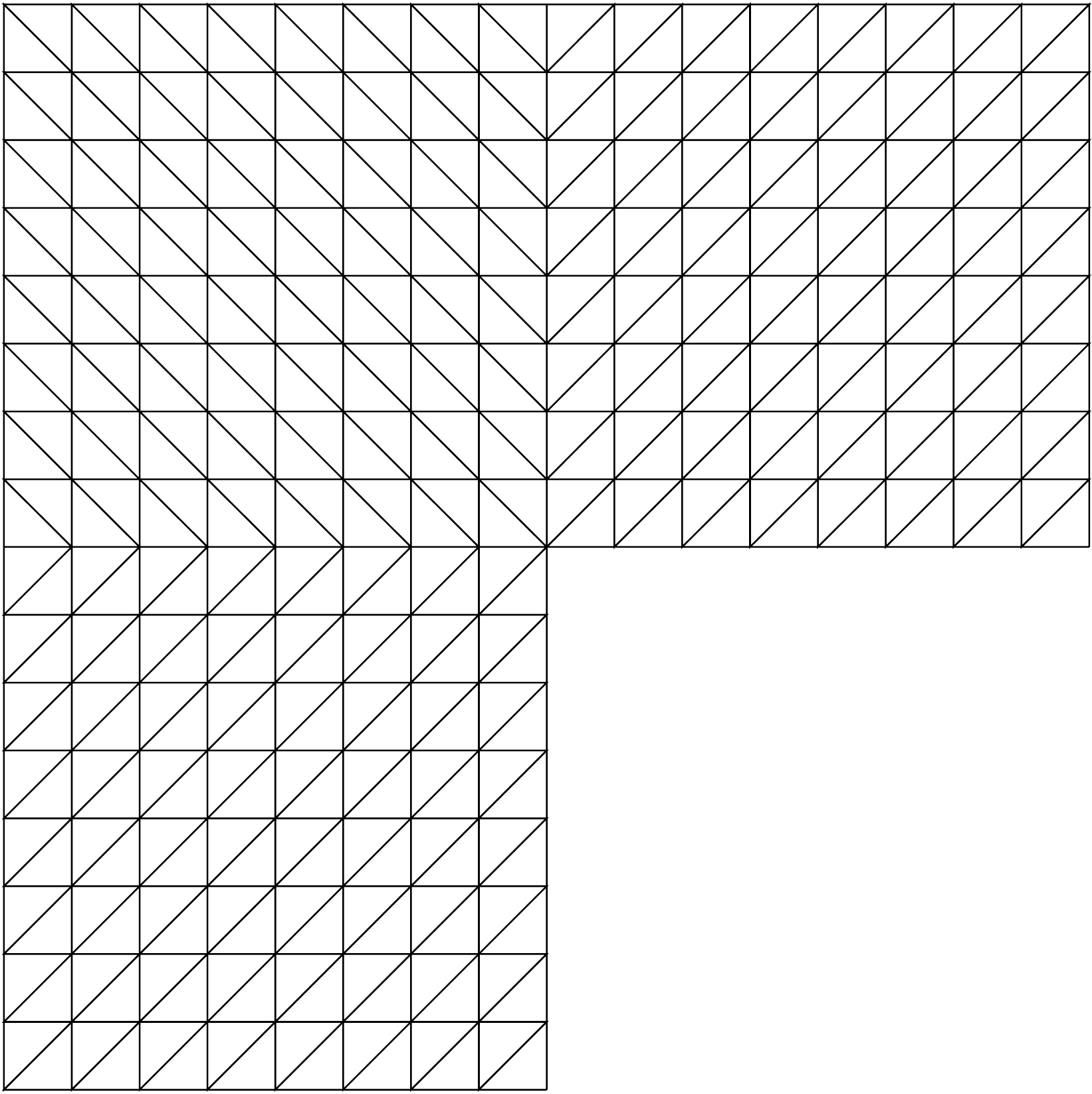}
\hspace{1cm}
\includegraphics[scale=0.35]{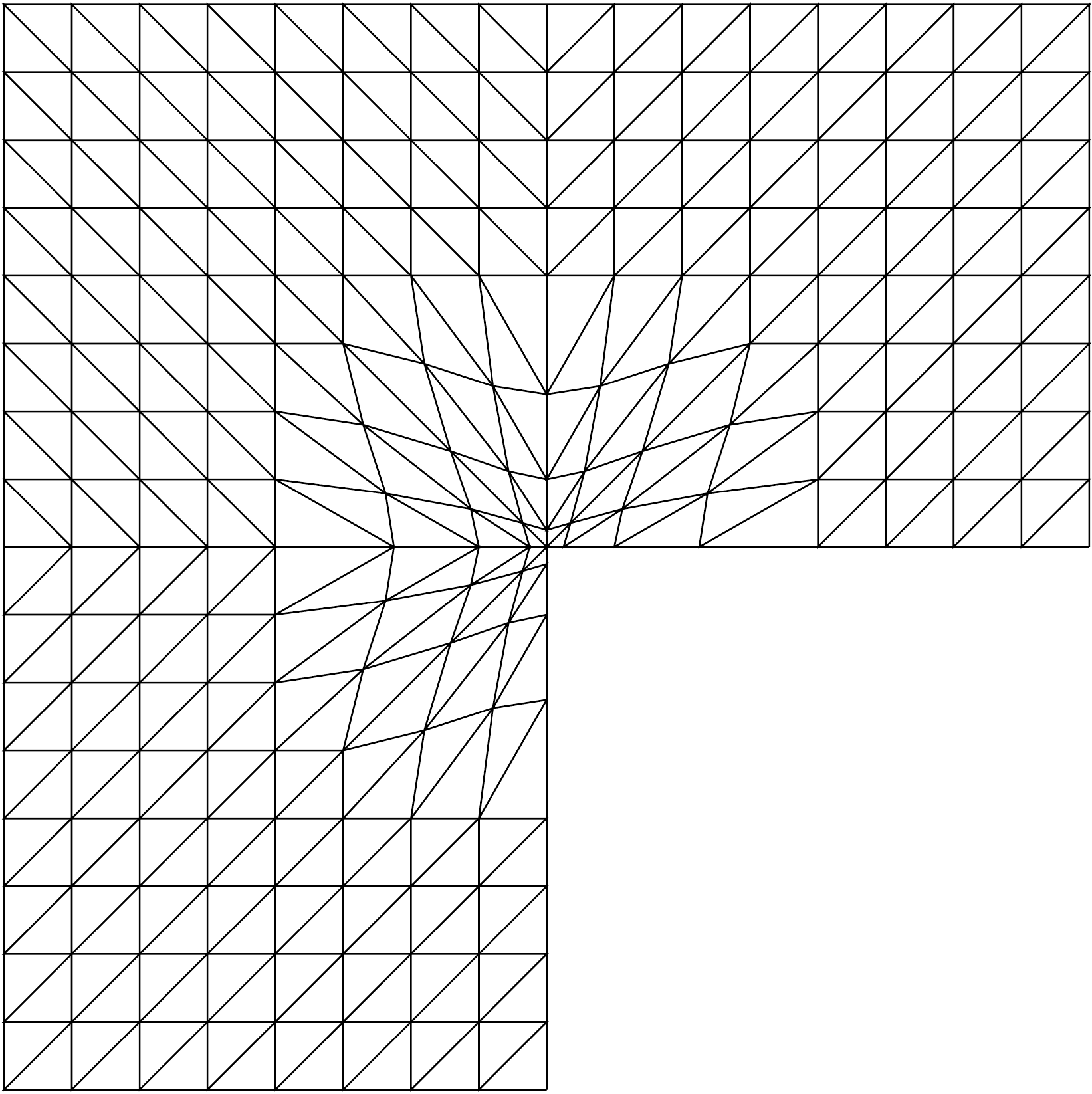}
\caption{$\O_{3\pi/2}$ with ungraded and graded mesh ($\mu=0.5,\,R=0.5$)}\label{gradedmesh}
\end{center}
\end{figure}
Figure \ref{solutions} shows the postprocessed control $\tilde u_h$ and the state $\yhb$ as the solution of the fully discretized optimal control problem (QP$_h$).\\
\begin{figure}
\begin{center}
\includegraphics[scale=0.35]{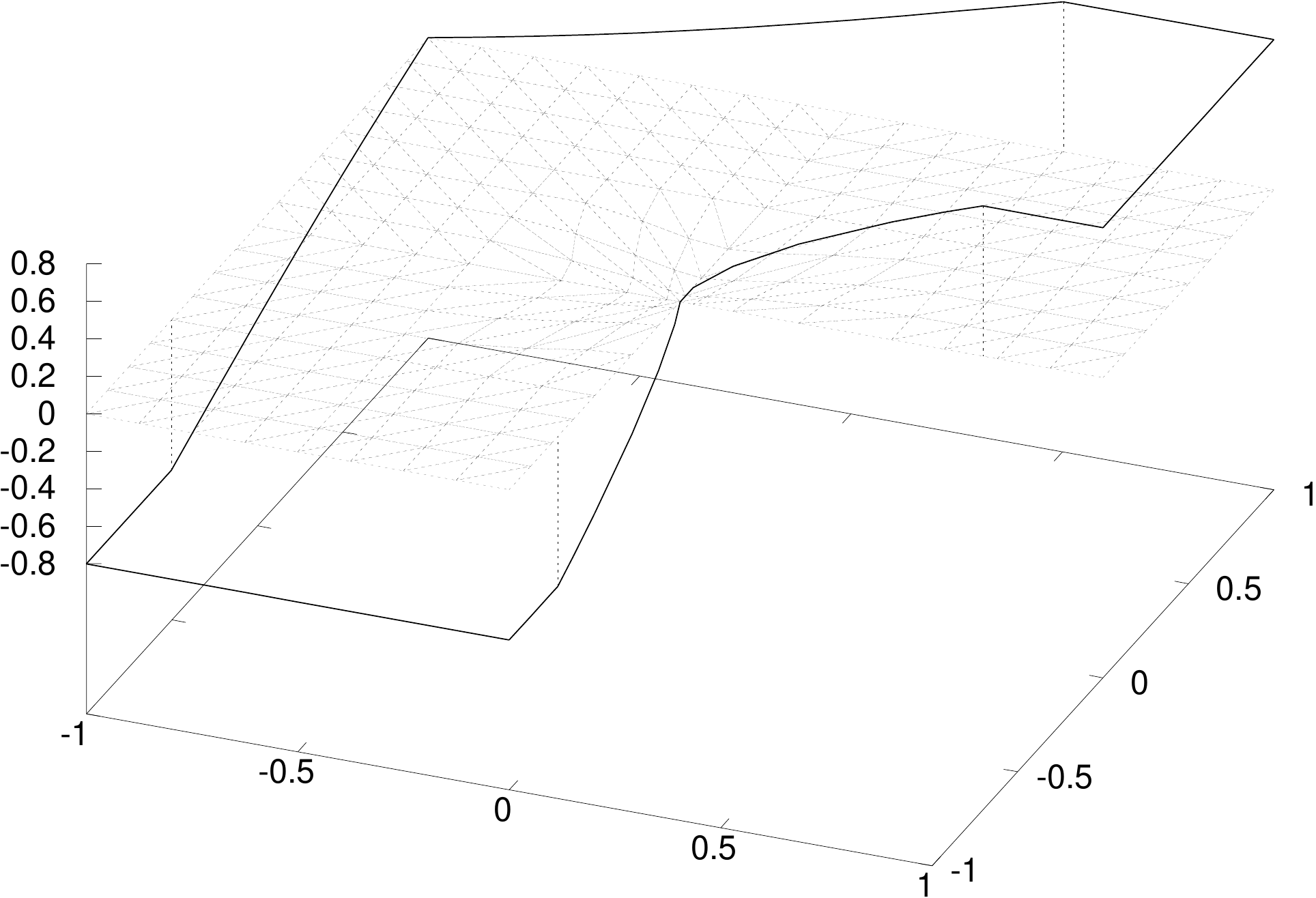}
\hspace{0.5cm}
\includegraphics[scale=0.35]{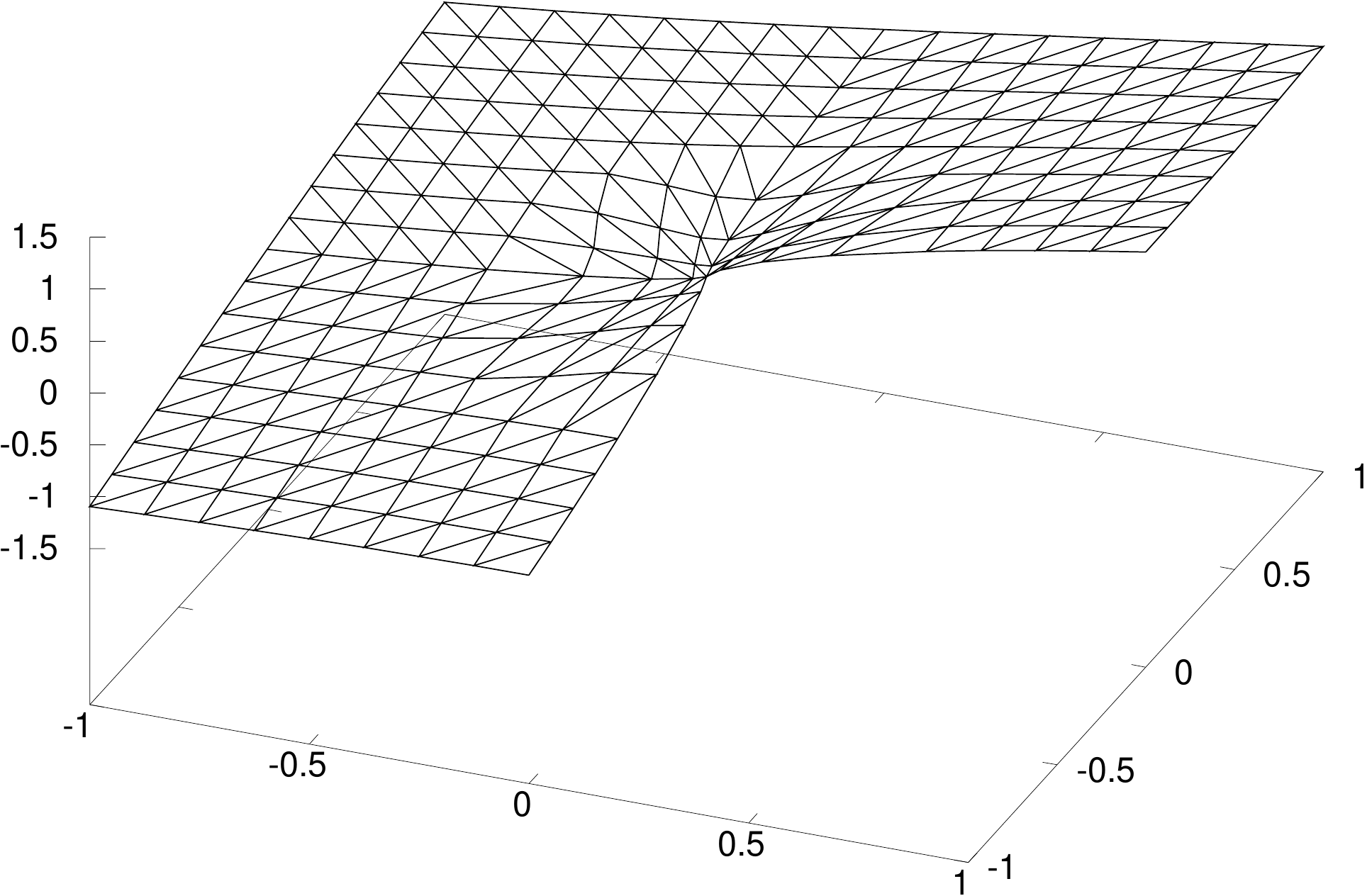}
\caption{$\tilde u_h$ and $\yhb$ on a graded mesh ($\mu=0.5,\,R=0.5$)}\label{solutions}
\end{center}
\end{figure}
In Table \ref{errorresults} one can find the computed errors $\lnorm{\ub-\tilde u_h}{\G_{3\pi/2}}$ and the experimental orders of convergence (EOC) once for uniform meshes ($\mu=1$) and for graded meshes with $\mu=0.5$.
According to Theorem 6.2. in \cite{MateosRoesch:2008}, we expect on uniform meshes a convergence rate $1/2+\lambda=1.16$ which is illustrated by our numerical results. Moreover, we can see that the numerical example confirms for a grading parameter $\mu=0.5<1/4+\lambda/2\approx 0.58$ the results proven in Theorem \ref{mainresult}.
\begin{table}
\begin{center}
\begin{tabular}{rrcccc}
\toprule
\mc{2}{c}{degrees of freedom}&\mc{2}{c}{$\mu=0.5$}&\mc{2}{c}{$\mu=1$}\\
$\O$&$\G$&$\|\ub-\tilde u_h\|$&EOC&$\|\ub-\tilde u_h\|$&EOC\\
\midrule
$133$&$48$&$1.23e-2$&$1.90$ &$2.75e-2$ &$1.10$\\
$481$&$92$&$3.64e-3$ &$1.92$&$1.37e-2$ &$1.12$\\
$1825$&$192$&$1.02e-3$ &$1.94$ &$6.58e-3$ &$1.13$\\
$7105$&$384$ &$2.76e-4$&$1.95$ &$3.07e-3$ &$1.13$\\
$28033$&$768$ &$7.27e-5$ &$1.96$ &$1.42e-3$ &$1.14$\\
$111361$&$1536$&$1.88e-5$ &$1.97$ &$6.46e-4$ &$1.14$\\
$443905$&$3072$&$4.83e-6$ &$1.97$ &$2.93e-4$ &$1.15$\\
$1772545$&$6144$&$1.23e-6$ &-&$1.32e-4$ &-\\
\bottomrule
\end{tabular}
\caption{$L^2(\G_{3\pi/2})$-error of the postprocessed control $\tilde u_h$}\label{errorresults}
\end{center}
\end{table}

\begin{appendix}
\section{Appendix}
\begin{lemma}\label{l1estimate_states}
Let Assumption~\ref{Psetting} be satisfied. Furthermore, let $\yhb=y_h(\bar u_h)$ and $y_h(R_h^{\ub}\ub)$ be the solutions of~\eqref{vardisPDE} w.r.t $\bar u_h$ and $R_h^{\ub}\ub$, respectively. Then there exists a mesh size $h_0>0$ such that for all $h<h_0$ the estimate
\begin{equation*}
\lnorm{\yhb-y_h(R_h^{\ub}\ub)}{\O}\le c\|\uhb-R_h^{\ub}\ub\|_{L^1(\Gamma)}
\end{equation*}
is valid.
\end{lemma}
\begin{proof}
Analogously to the beginning of the proof of Lemma \ref{yhu-yhrhu}, we introduce a dual auxiliary problem and its discrete counterpart by: let $\phi\in H^1(\O)$ be the unique solution of
\[
  a(\phi,v)+\into{\alpha\phi v}=\into{(\yhb-y_h(R_h^{\ub}\ub))v}\quad\quad\forall v\in H^1(\O),
\]
with
\[
  \alpha(x)=\left\{
  \begin{aligned}
  \frac{d(x,\yhb(x))-d(x,y_h(R_h^{\ub}\ub)(x))}{\yhb(x)-y_h(R_h^{\ub}\ub)(x)},&\quad\text{if }\yhb(x)-y_h(R_h^{\ub}\ub)(x)\not=0\\
  0,&\quad\text{otherwise}
  \end{aligned}
  \right..
\]
Due to the approximation properties of $R_h^{\ub}$, the uniform convergence of $\bar u_h$ to $\bar u$ and the monotonicity of $d$ according to Assumption~\ref{Psetting} one can easily check that there exists a $h_0>0$ such that the problem is well-posed for all $h<h_0$. The corresponding discrete counterpart $\phi_h\in V_h$ is the unique solution of the problem
\[
  a(\phi_h,v_h)+\into{\alpha\phi_h v_h}=\into{(\yhb-y_h(R_h^{\ub}\ub))v_h}\quad\forall v_h\in V_h.
\]
By means of $\yhb(\ub), y_h(R_h^{\ub}\ub)\in V_h$ being solutions of \eqref{vardisPDE} and the definition of $\alpha$, we derive
\[
  \begin{split}
  \lnorm{\yhb(\ub)-y_h&(R_h^{\ub}\ub)}{\O}^2=a(\phi_h, \yhb(\ub)-y_h(R_h^{\ub}\ub))+\into{\alpha\phi_h(\yhb(\ub)- y_h(R_h^{\ub}\ub))}\\
  &=a(\yhb(\ub)-y_h(R_h^{\ub}\ub),\phi_h)+\into{(d(x,\yhb(\ub))-d(x,y_h(R_h^{\ub}\ub))\phi_h}\\
  &=\intg{(\uhb-R_h^{\ub}\ub)\phi_h}.
  \end{split}
\]
We continue by the estimates
\[
    \begin{aligned}
      \intg{(\uhb-R_h^{\ub}\ub)\phi_h}&\le \|\uhb-R_h^{\ub}\ub\|_{L^1(\G)}\|\phi_h\|_{L^\infty(\G)}\\
      &\le \|\uhb-R_h^{\ub}\ub\|_{L^1(\G)}(\|\phi_h-\phi\|_{L^\infty(\O)}+\|\phi\|_{L^\infty(\O)})\\
      &\le c\|\uhb-R_h^{\ub}\ub\|_{L^1(\G)}\left(h^{1/2-\varepsilon}+1\right)\lnorm{\yhb-y_h(R_h^{\ub}\ub)}{\O},
    \end{aligned}
\]
where a standard $L^\infty(\O)$-error estimate (see e.g.~\eqref{fe-error-inf}) and Lemma~\ref{pro:reglinearPDE} together with the embedding $H^{3/2}(\O)\hookrightarrow L^\infty(\O)$ were used. Thus, the assertion is proven.
\end{proof}
\begin{lemma}\label{l1estimate_linearizedstates}
Suppose that the assumptions (A\ref{A4}) and (A\ref{A6}) are fulfilled. Moreover, let $y_h^{v}\in V_h$ be the unique solution of \eqref{Gh'} for a given discrete state $y_h$ w.r.t. the right hand side $v$. Then the estimate
\begin{equation*}
  \lnorm{y_h^{v}}{\O}\le c\|v\|_{L^1(\Gamma)}
\end{equation*}
holds true.
\end{lemma}
\begin{proof}
The proof can be done analogously to the proof of Lemma \ref{l1estimate_states} introducing an appropriate dual problem.
\end{proof}
\end{appendix}

\bibliographystyle{plain}
\bibliography{literatur_neumann.bib}

\end{document}